\documentclass[12pt,reqno,a4paper]{amsart}
\usepackage{euscript}

\newtheorem{proposition}{Proposition}
\newtheorem{lemma}{Lemma}

\newtheorem{remark}{Remark}
\newtheorem{definition}{Definition}

\DeclareMathOperator{\Rank}{Rank}

\def\R{\mathbb{R}}

\usepackage{fullpage,amsmath,amssymb,amsthm,enumitem,multicol,listings,color,graphicx,wrapfig,soul}
\usepackage[lofdepth,lotdepth]{subfig}
\usepackage{etoolbox}
\makeatletter
\patchcmd{\@verbatim}
  {\verbatim@font}
  {\verbatim@font\footnotesize}
  {}{}
\makeatother

\title{Optimal linear responses for Markov chains and stochastically perturbed dynamical systems}
\author{Fadi Antown, Davor Dragi\v cevi\' c, and Gary Froyland}
\date{5 September 2017}
\begin{document}
\begin{abstract}
The linear response of a dynamical system refers to changes to properties of the system when small external perturbations are applied.
We consider the little-studied question of selecting an optimal perturbation so as to (i) maximise the linear response of the equilibrium distribution of the system, (ii) maximise the linear response of the expectation of a specified observable, and (iii) maximise the linear response of the rate of convergence of the system to the equilibrium distribution.
We also consider the inhomogeneous or time-dependent situation where the governing dynamics is not stationary and one wishes to select a sequence of small perturbations so as to maximise the overall linear response at some terminal time.
We develop the theory for finite-state Markov chains, provide explicit solutions for some illustrative examples, and
 numerically apply our theory to stochastically perturbed dynamical systems, where the Markov chain is replaced by a matrix representation of an approximate annealed transfer operator for the random dynamical system.
\end{abstract}
\maketitle

\section{Introduction}
The notion of linear response crosses many disciplinary boundaries in mathematics and physics.
At a broad level, one is interested in how various quantities respond to small perturbations in the dynamics.
Historically, this response is often studied through the changes in the equilibrium probability distribution of the system.
In certain cases, if the governing dynamics varies according to a parameter, one can formally express the change in the equilibrium probability distribution as a derivative of the governing dynamics with respect to this parameter.

Finite state Markov chains are one of the simplest settings in which to study formal linear response, and early work includes Schweitzer \cite{schweitzer} who stated response formulae for invariant probability distributions under perturbations of the governing $n\times n$ stochastic matrix $P$.
The perturbations in \cite{schweitzer} were either macroscopic or infinitesimal, and in the latter case the response was expressed as a derivative.
Linear response has been heavily studied in the context of smooth or piecewise smooth dynamical systems.
In the case of uniformly (and some nonuniformly) hyperbolic dynamics, there is a distinguished equilibrium measure, the Sinai-Bowen-Ruelle (SBR) measure, which is exhibited by a positive Lebesgue measure set of initial conditions.
Ruelle \cite{ruellemap} developed response formulae for this SBR measure for uniformly hyperbolic maps;  this was extended to partially hyperbolic maps by Dolgopyat \cite{dolgopyat} and to uniformly hyperbolic flows \cite{ruelleflow,butterley}.
Modern approaches to proving linear response, such as \cite{gouezel06,butterley,gouezel08} do not rely on coding techniques as in \cite{ruellemap}, but work directly with differentiability properties of transfer operators acting on anisotropic Banach spaces.
For expanding and/or one-dimensional dynamics, linear response for unimodal maps \cite{baladismania} and intermittent maps \cite{bahsoun,baladitodd} has been established;  see also the surveys \cite{liverani_notes,baladiicm}.
Linear response results for stochastic systems using Markov (transfer) operator techniques have also been developed \cite{hairermajda} and linear response for inhomogeneous Markov chains have also been considered \cite{mackay}.
There is a great deal of activity concerning the linear (or otherwise) response of the Earth's climate system to external perturbations \cite{abramov_new_2009,chekroun,ragone_new_2015}, and there have been recent extensions to the linear response of multipoint correlations of observables \cite{lucariniwouters}.

Much of the theoretical focus on linear response has been on establishing that for various classes of systems, there \emph{is} a principle of linear response.
Our focus in this work is in a much less studied direction, namely, determining those perturbations that lead to \emph{maximal} response.
This problem of \emph{optimizing} response is of intrinsic mathematical interest, and also has practical implications:  not only is it important to establish the maximal sensitivity of a system to small perturbations, but it is also of great interest to identify those specific perturbations that provoke a maximal system response.
For example, a common application of linear response is the response of various models of the Earth's climate to external (man-made) forcing.
Mitigation strategies ought to specifically avoid those perturbations that lead to large and unpredictable responses, and it is therefore important to be able to efficiently identify these maximal response perturbations.
This important avenue of research has relatively few precendents in the literature.
One exception is \cite{wangetal} who consider optimal control of Langevin dynamics; using a linear response approach, they apply a gradient descent algorithm to minimise a specified linear functional.

The questions we ask are:  (i) What is the perturbation that provokes the greatest linear response in the equilibrium distribution of the dynamics? (ii) What is the perturbation that maximally increases the value of a specified linear functional?  (iii) What is the perturbation that has the greatest impact on the rate of convergence of the system to equilbrium?
We answer these questions in the setting of finite state Markov chains, including the inhomogeneous situation.
Question (i) turns out to be the most difficult because of its non-convex nature: we are maximising (not minimising) an $\ell_2$ norm.
We develop an efficient numerical approach, based on solving an eigenproblem, which exploits sparsity of the transition matrix when present.
We are able to solve questions (ii) and (iii) in closed form, following some preliminary computations (solving a linear system and solving an eigenproblem, respectively).

In the numerics section, we apply our results to Ulam discretisations of stochastically perturbed dynamical systems in one dimension.
These Ulam discretisations are large sparse stochastic matrices and thus our previous results readily apply.
We limit ourselves to one-dimensional examples to provide a clearer presentation of the results, but there is no obstacle to carrying out these computations in two- or three-dimensional systems.
The types of dynamical systems that can be considered are of the following forms:
\begin{enumerate}
\item 
One has deterministic dynamics $T:X\to X$, $X\subset \mathbb{R}^d$ with stochastic perturbations that are an integral part of the model.
There is a background i.i.d.\ stochastic process $\{\xi_n\}_{n=0}^\infty$, with the random variables $\xi:\Omega\to X$ 
 creating the perturbed dynamics $x_{n+1}=T(x_n)+\xi_n$, $n\ge 0$.
\item
One has a collection of deterministic maps $\{T_{\omega_n}\}_{n=0}^\infty$ which are composed in an i.i.d.\ fashion: $\cdots\circ T_{\omega_k}\cdots\circ T_{\omega_2}\circ T_{\omega_1}$, where $\omega\in\Omega$ is distributed according to a probability measure $\mathbb{P}$ on $\Omega$.
In the special case where $\Omega\subset X\subset \mathbb{R}^d$ and $T_{\omega_i}x=Tx+\omega_i$ for some fixed $T$, this situation coincides with the previous one.
\end{enumerate}
In both cases, one forms an annealed transfer operator $\mathcal{P}f=\int_\Omega \mathcal{P}_{T_{\omega}}\ d\mathbb{P}(\omega)$.
If $\Omega\subset X$ and $\mathbb{P}$ has density $q$ with respect to Lebesgue measure we may write $$\mathcal{P}f(x)=\int_\Omega f(y)q(x-T_{\omega}y)\  d\ell(y).$$
Under mild conditions (see section \ref{sect:numeric}) $\mathcal{P}:L^2(X)\to L^2(X)$ is compact and has a unique fixed point $h$, which can be normalised as $\int_X h(x)\ dx=1$ to form an invariant density of the annealed stochastic dynamics.
One can ask how to alter the stochastic kernel $q$, which governs the stochastically perturbed dynamical system, to achieve maximal linear responses.


The first question we consider is ``how should the new stochastic process be changed in order to produce the greatest linear response to the $L^2$ norm of $h$?''.
Given a small change in the kernel $q$ we obtain a new invariant measure $\mu'$.
Denote $\delta\mu=\mu'-\mu$ and $\delta h$ the density of $\delta\mu$ with respect to Lebesgue;  we wish to select $q$ so as to provoke the greatest change $\delta h$ in an $L^2$ sense.
One motivation for this question is to determine the maximal sensitivity for \textit{all} normalised observables $c\in L^2(X)$.
One has $|\mathbb{E}_{\delta\mu}(c)|\le \|c\|_{L^2}\cdot \|\delta h\|_{L^2}$ and thus $\sup_{\|c\|_{L^2}\le 1} |\mathbb{E}_{\delta\mu}(c)|\le \|\delta h\|_{L^2}$.
In certain situations, if the density $h$ is important in an energy sense, then the $L^2$ norm of the response is important from an energy point of view.
In a recent article \cite{galatolopollicott} consider expanding maps of the interval and determine the perturbation of least (Sobolev-type) norm which produces a \emph{ given linear response}.
In contrast, here we study the question of finding the perturbation that produces the linear response of greatest size.

Second, we consider the problem of maximising linear response of a specific observable $c:X\to \mathbb{R}$ to a change in the stochastic perturbations.
Given a small change in the kernel $q$ we obtain a new invariant measure $\mu'$, and we compare $\mathbb{E}_\mu(c)$ with $\mathbb{E}_{\mu'}(c)$.
How should the new stochastic process be changed in order that the expectation $\mathbb{E}_\mu(c)$ increases at the greatest rate?
Put another way, what is the most ``$c$-sensitive direction'' in the space of stochastic perturbations?

Third, we ask which perturbation of the kernel $q$ produces the greatest change in the rate of convergence to the equilibrium measure of the stochastic process.
This rate of convergence is determined by the magnitude of the second eigenvalue $\lambda_2$ of the transfer operator $\mathcal{P}$ and we determine the perturbation that pushes the eigenvalue farthest from the unit circle.
Related perturbative approaches include \cite{FrSa}, where the mixing rate of (possibly periodically driven) fluid flows was increased by perturbing the advective part of the dynamics and solving a linear program; \cite{FGTW}, where similar kernel perturbation ideas were used to drive a nonequilibrium density toward equilibrium by solving a convex quadratic program with linear constraints;  and \cite{grover_elam} where a governing flow is perturbed deterministically so as to evolve a specified initial density into a specified final density over a fixed time duration, with the perturbation determined as the numerical solution of a convex optimisation problem.
In the current setting, our perturbation acts on the stochastic part of the dynamics and we can find a solution in closed form after some preliminary computations.

An outline of the paper is as follows.
In Section \ref{sect:prelim} we set up the fundamentals of linear response in finite dimensions.
Section \ref{sect:l2} tackles the problem of finding the perturbation that maximises the linear response of the equilibrium measure in an $\ell_2$ sense.
We first treat the easier case where the transition matrix for the Markov chain is positive, before moving to the situation of a general irreducible aperiodic Markov chain.
In both cases we provide sufficient conditions for a unique optimum, and present explicit algorithms, including \verb"MATLAB" code to carry out the necessary computations.
We illustrate these algorithms with two simple analytic examples, which we carry through the paper.
Section \ref{sect:expectation} solves the problem of maximising the linear response of the expection with respect to a particular observable, while section \ref{sect:rate} demonstrates how to find the perturbation that maximises the linear response of the rate of convergence to equilibrium.
In both of these sections, we provide sufficient conditions for a unique optimum, present explicit algorithms, code, and treat two analytic examples.
Section \ref{sect:sequential} considers the linear response problems for a finite sequence of (in general different) stochastic transition matrices.
Section \ref{sect:numeric} applies the theory of Sections \ref{sect:l2}--\ref{sect:rate} to stochastically perturbed one-dimensional chaotic maps.
We develop a numerical scheme to produce finite-rank approximations of the transfer (Perron-Frobenius) operators corresponding to the stochastically perturbed maps.
These finite-rank approximations have a stochastic matrix representation, allowing the preceding theory to be applied.

\section{Notation and setting}\label{sect:prelim}

We follow the notation and initial setup of \cite{lucarini2016response}.
Consider a column stochastic transition matrix $M=(M_{ij})\in\R^{n\times n}$ of a mixing Markov chain on a finite state space $\{1,\dots,n\}$. More precisely, we assume that $M$ satisfies:
\begin{enumerate}
\item $0\leq M_{ij}\leq 1$ for every $i,j \in \{1,\dots,n \}$;
\item $\sum_{i=1}^n M_{ij} = 1$  for every  $j\in \{1,\dots,n\}$;
\item  there exists $N \in \mathbb N$ such that $M^N_{ij}>0$ for every $i,j \in \{1,\dots,n \}$.
\end{enumerate}
Let $\textbf{h}_M= (h_1,\dots,h_n)^\top\in\R^n$ denote the invariant probability vector of $M$, i.e. the probability vector such that $M\textbf{h}_M = \textbf{h}_M$. We note that the
existence and the uniqueness of $\textbf{h}_M$ follow from the above assumptions on $M$.
Moreover,  let us consider  perturbations of $M$ of the form $M+\varepsilon m$, where
$\varepsilon \in \mathbb R$ and $m \in \R^{n\times n}$. In order to ensure that  $M+\varepsilon m$ is also a column stochastic matrix, we need to impose some conditions on $m$ and $\varepsilon$.
For a fixed $m=(m_{ij})\in\R^{n\times n}$, we require that \begin{equation}\label{o} \sum_{i=1}^n m_{ij} = 0 \quad \text{for every $j\in \{1, \ldots, n\}.$}\end{equation}
Furthermore,  we assume that $\varepsilon\in [\varepsilon_-,\varepsilon_+]$ and  $\varepsilon_-<\varepsilon_+$,   where
\[
\varepsilon_+ := \max_{\varepsilon}\{\varepsilon\in\R: M_{ij}+\varepsilon m_{ij}\geq 0\  \text{for every  $i,j \in \{1,\ldots,n \}$} \}
\]
and
\[
\varepsilon_- := \min_{\varepsilon}\{\varepsilon\in\R: M_{ij}+\varepsilon m_{ij}\geq 0\ \text{for every $i,j \in \{1,\ldots,n\}$} \}.
\]
Let us denote the invariant probability vector of the perturbed transition matrix $M+\varepsilon m$ by $\textbf{h}_{M+\varepsilon m}$.
We remark that by decreasing $\left[\epsilon_-, \varepsilon_+\right]$ we can ensure that the invariant probability vector $\textbf{h}_{M+\varepsilon m}$ remains unique.
If we write
\begin{equation}\label{series-expansion}
 \textbf{h}_{M+\varepsilon m} = \textbf{h}_M +\sum_{j=1}^{\infty}\varepsilon^j \textbf{u}_j,
\end{equation}
where $\varepsilon\in\R$ is close to $0$, then $\textbf{u}_1$ is defined as the \textit{linear response}
of the invariant probability vector $\textbf{h}_M$ to the perturbation $\varepsilon m$.

By summing the entries of both sides of (\ref{series-expansion})
and comparing $\varepsilon$ orders, we must have that the column sum of the vector $\textbf{u}_1$ is zero. On the other hand, since $\textbf{h}_{M+\varepsilon m}$
is an invariant probability vector of $M+\varepsilon m$, we have that
\begin{equation}\label{Response-Series}
(M+\varepsilon m)\bigg{(}\textbf{h}_M +\sum_{j=1}^{\infty}\varepsilon^j \textbf{u}_j \bigg{)} = \textbf{h}_M +\sum_{j=1}^{\infty}\varepsilon^j \textbf{u}_j.
\end{equation}
By expanding the left-hand side of~\eqref{Response-Series}, we obtain that
\begin{equation*}
(M+\varepsilon m)\bigg{(}\textbf{h}_M +\sum_{j=1}^{\infty}\varepsilon^j \textbf{u}_j \bigg{)} = \textbf{h}_M +\varepsilon (M\textbf{u}_1+m\textbf{h}_M)+O(\varepsilon^2).
\end{equation*}
Hence, it follows from~\eqref{series-expansion} and~\eqref{Response-Series} that the linear response $\textbf{u}_1$ satisfies equations
\begin{equation}\label{Sing-Lin-Sys}
(\text{Id}-M)\textbf{u}_1 = m\textbf{h}_M
\end{equation}
and
\begin{equation}\label{zero-sum-u}
\textbf{1}^\top \textbf{u}_1 = 0,
\end{equation}
where $\textbf{1}^\top = (1,\dots,1)\in\R^n$.
We note that the matrix $\text{Id}-M$ is singular since $1$ is an eigenvalue of $M$ (with the corresponding eigenvector $\textbf{h}_M$).
However, the restriction of  $\text{Id}-M$ to the subspace $\R^n_0$ of $\R^n$ spanned by all other eigenvectors of $M$, is invertible. We note that $\R^n_0$ consists of all
vectors of column sum zero. Indeed, this  follows from the fact that  $\textbf{1}^\top$ is a left eigenvector of $M$
corresponding to eigenvalue 1 and consequently,  it is orthogonal to all right eigenvectors of $M$ except for $\textbf{h}_M$.
Alternatively, by Theorem 2 from \cite{kemeny1981generalization} we can conclude that  the linear system (\ref{Sing-Lin-Sys})-(\ref{zero-sum-u}) has the unique solution given by
\begin{equation}\label{Lin-Resp}
\textbf{u}_1 = Qm\textbf{h}_M,
\end{equation}
where
\begin{equation}\label{Fun-Mat-MC}
Q = \left(\text{Id}-M+\textbf{h}_M\textbf{1}^\top\right)^{-1}.
\end{equation}
The matrix $Q$ is called the \emph{fundamental matrix} of the transition matrix $M$.
We note that the matrix $Q$ is the so-called generalized inverse of $\text{Id}-M$, which means that  it satisfies \[(\text{Id}-M)Q(\text{Id}-M) = \text{Id}-M. \]
We refer to~\cite{hunter1969moments} for details.
In the rest of the paper, we will denote $\textbf{h}_M$ simply by $\textbf{h}$.

\section{Maximizing the Euclidean norm of the linear response of the invariant measure}
\label{sect:l2}
Our aim in this section is to find the perturbation $m$ that will maximise the Euclidean norm of the linear response. We will start by considering the case when $M$ has all positive entries and later we will deal with the general case when $M\in\R^{n\times n}$ is the transition matrix of an arbitrary mixing Markov chain.

\subsection{The Kronecker Product}
In this subsection, we will briefly introduce the Kronecker product and some of its basic properties. These results will be used to convert some of our optimization problems into simpler, smaller, and more numerically stable forms.
\begin{definition}
Let $A = (\mathbf{a}_1|\dots |\mathbf{a}_n) = (a_{ij})_{ij}$ be an $m\times n$ matrix and $B$ a $p\times q$ matrix. The $mp\times nq$ matrix given by
$$\left(\begin{array}{ccc}
a_{11}B & \dots & a_{1n}B \\
\vdots &  & \vdots \\
a_{m1}B & \dots & a_{mn}B
\end{array}\right)$$
is called the \emph{Kronecker product} of $A$ and $B$ and is denoted by $A\otimes B$. Furthermore, the \emph{vectorization} of $A$ is given by the vector
$$\widehat{A} :=\left(\begin{array}{c}
\mathbf{a}_1 \\
\vdots \\
\mathbf{a}_n
\end{array} \right) \in \R^{mn}.$$
\end{definition}
The following result collects some basic properties of the Kronecker product.
\begin{proposition}[{\cite{laub2005matrix}}]\label{Prop-Kron-Prod}
Let $A,B,C,D$ be $m\times n$, $p\times q$, $n\times n$ and $q\times q$ matrices respectively, and let $\alpha\in\R$. Then, the following identities hold:
\begin{enumerate}[label=(\roman*)]
\item $(A\otimes B)(C\otimes D) = AC\otimes BD$;
\item $\alpha A = \alpha\otimes A = A\otimes\alpha$;
\item $(A\otimes B)^\top = A^\top \otimes B^\top$, where $A^\top$ denotes the transpose of $A$;
\item $\Rank(A\otimes B) = (\Rank(A))\cdot (\Rank(B))$;
\item let $\lambda_1,\dots,\lambda_n$ be the eigenvalues of $C$ and $\mu_1,\dots,\mu_q$ be the eigenvalues of $D$. Then, the $nq$ eigenvalues of $C\otimes D$ are given by $\lambda_i\mu_j$, for $i=1,\dots,n$ and $j=1,\dots,q$. Moreover, if $\textbf{x}_1\dots,\textbf{x}_n$ are linearly independent right eigenvectors of $C$ corresponding to $\lambda_1,\dots,\lambda_n$ and $\textbf{y}_1\dots,\textbf{y}_q$ are linearly independent right eigenvectors of $D$ corresponding to $\mu_1,\dots,\mu_q$, then $\textbf{x}_i\otimes \textbf{y}_j$ are linearly independent right eigenvectors of $C\otimes D$ corresponding to $\lambda_i\mu_j$;
\item for any $n\times p$ matrix $E$, we have  $$\widehat{AEB} = (B^\top\otimes A)\widehat{E}.$$
\end{enumerate}
\end{proposition}

\subsection{An alternative formula for the linear response of the invariant measure}

As a  first application of the Kronecker product, we give  an alternative formula for the linear response (\ref{Lin-Resp}). Using Proposition \ref{Prop-Kron-Prod}(vi) and noting that $Qm\textbf{h}$ is an $n\times 1$ vector, we can write
\begin{equation}\label{lin-resp-formula-vectz}
Qm\textbf{h} =\widehat{Qm\textbf{h}} =  \left(\textbf{h}^\top\otimes Q\right)\widehat{m}=W\widehat{m},
\end{equation}
where $W = \textbf{h}^\top\otimes Q$. Note that $\textbf{h}^\top$ is of dimension $1\times n$ and $Q$ is of dimension $n\times n$. Thus, the dimension of $W$ is $n\times n^2$.
We now have two equivalent formulas for the linear response: \eqref{Lin-Resp} in terms of the matrix $m$  and \eqref{lin-resp-formula-vectz} in terms of the vectorization $\widehat{m}$. In sections \ref{Section-Positive-M} and \ref{Section-General-M} of the paper, the formula \eqref{lin-resp-formula-vectz} will be predominately used.

\subsection{Positive transition matrix $M$}\label{Section-Positive-M}
We first suppose that the transition matrix is positive, i.e.  $M_{ij}>0$ for every $i, j\in \{1, \ldots, n\}$ (section \ref{Section-General-M} handles general stochastic $M$). In this subsection, we will find the perturbation $m$ that maximises the Euclidean norm of the linear response. More precisely, we consider the following optimization problem:

\begin{eqnarray}
\label{obj}
\max_{m\in\R^{n\times n}} &&\|Qm\textbf{h}\|_2^2\\
\label{stoch}\mbox{subject to} && m^\top\mathbf{1} =\mathbf{0}\\
\label{norm}&&\|m\|_F^2-1=0,
\end{eqnarray}
where $\|\cdot\|_2$ is the Euclidean norm and $\|\cdot\|_F$ is the Frobenius norm defined by
$$\|A\|^2_F = \sum_i\sum_j |a_{ij}|^2, \quad \text{for  $A=(a_{ij})$.}$$ We note that the constraint~\eqref{stoch} corresponds to the condition~\eqref{o}, while~\eqref{norm} is imposed to ensure the  existence (finiteness) of the solution.
Furthermore,
we observe that a solution to the above optimization problem exists since we are maximising a continuous function on a compact subset of $\R^{n\times n}$.

\subsubsection{Reformulating the problem (\ref{obj})-(\ref{norm}) in vectorized form:}
We begin by reformulating  the problem (\ref{obj})-(\ref{norm}) in order to obtain an equivalent optimization problem over a space of vectors as opposed to a space of matrices.
Using \eqref{lin-resp-formula-vectz}, we can write the objective function in~\eqref{obj} as $\|W\widehat{m}\|_2^2$.
Similarly, we can rewrite the constraint \eqref{stoch} in terms of $\widehat{m}$. More precisely, we have the following auxiliary result. Let $\text{Id}_n$ denote an identity matrix of dimension $n$.

\begin{lemma}\label{lem1}
The constraint (\ref{stoch}) can be written in the form $A\widehat{m}=\bf{0}$, where $A$ is an $n\times n^2$ matrix given by
\begin{equation}\label{equ-constraint}
A = \text{Id}_n\otimes \mathbf{1}^\top.
\end{equation}
\end{lemma}
\begin{proof}
We have that $\textbf{1}^\top m$ is a $1\times n$ vector and thus $\widehat{\textbf{1}^\top m} =m ^\top\textbf{1}$. Furthermore, using Proposition \ref{Prop-Kron-Prod}(vi) we have that
$$ m^\top\textbf{1} = \widehat{\textbf{1}^\top m} = \widehat{\textbf{1}^\top m \text{Id}_n} = \left(\text{Id}_n\otimes \textbf{1}^\top\right)\widehat{m} = A\widehat{m}. $$
Finally, we note that since $\text{Id}_n$ is an $n\times n$ matrix and $\textbf{1}^\top$ is an $1\times n$ vector, we have that $A$ is an $n\times n^2$ matrix.
\end{proof}
We also observe that $$\|m\|_F^2 = \sum_i\sum_j |m_{ij}|^2 = \|\widehat{m}\|_2^2.$$ Consequently, we can rewrite constraint~\eqref{norm} in terms of the Euclidean norm of the vector $\widehat{m}$.
Our optimization problem~\eqref{obj}-\eqref{norm} is therefore equivalent to the following:
\begin{eqnarray}
\label{obj2}
\max_{\widehat{m}\in\mathbb{R}^{n^2}} &&\|W\widehat{m}\|_2^2\\
\mbox{subject to} && A\widehat{m}=\mathbf{0}\label{stoch-A}\\
&&\|\widehat{m}\|_2^2-1=0\label{2-norm}.
\end{eqnarray}

\subsubsection{Reformulating the problem (\ref{obj2})-(\ref{2-norm}) to remove constraint (\ref{stoch-A}):}\label{Reform-to-final-soln}
Finally, we reformulate the problem (\ref{obj2})-(\ref{2-norm}) in order to solve it as an eigenvalue problem.
Consider the subspace $V$ of $\R^{n^2}$ given by
\begin{equation}\label{subspace-V}
V = \left\lbrace\textbf{x}\in\R^{n^2}:A\textbf{x} = \textbf{0}\right\rbrace.
\end{equation}
 We can write $V$ as
\begin{equation}\label{subspace-V-span}
V = \text{span}\{\textbf{v}_1,\dots,\textbf{v}_{\ell}\},
\end{equation}
where $\textbf{v}_k \in\mathbb{R}^{n^2}$,  $k\in \{1,\ldots,\ell \}$ form a basis of $V$. Note that  $\ell = n^2-n$. Indeed, it follows from Proposition~\ref{Prop-Kron-Prod}(iv)
and~\eqref{equ-constraint} that $\Rank(A) = \Rank(\text{Id}_n)\Rank(\textbf{1}^\top) = n$,
and thus by the rank-nullity theorem we have that $\ell=n^2-n$.

Taking $\widehat{m} \in V$ and writing \begin{equation}\label{E} E = (\textbf{v}_1|\dots |\textbf{v}_{\ell}),\end{equation}  we conclude  that there exists
a unique $\boldsymbol\alpha\in\mathbb{R}^{\ell}$ such that $\widehat{m} = E\boldsymbol\alpha$.
Hence, $\boldsymbol\alpha = E^+\widehat{m}$, where $E^+$ denotes the left inverse of $E$ given by
\[
 E^+:= (E^\top E)^{-1}E^\top.
\]
Note that since
 $E$ has full rank, we have that $E^\top E$ is non-singular (see p.43, \cite{ben2003generalized}) and therefore  $E^+$ is well-defined.
Using the above identities, we obtain that
\begin{equation}\label{Proj-on-Null}
W\widehat{m} = WE\boldsymbol\alpha = WEE^+\widehat{m}.
\end{equation}
Let \begin{equation}\label{eqU} U = WEE^+. \end{equation} Since the only assumption on $\widehat{m}$ was that $\widehat{m}\in V$, the problem (\ref{obj2})-(\ref{2-norm}) is equivalent to the following:
\begin{eqnarray}
\label{obj3}
\max_{\widehat{m}\in\mathbb{R}^{n^2}} &&\|U\widehat{m}\|_2^2\\
\mbox{subject to} &&\|\widehat{m}\|_2^2-1=0\label{norm3}.
\end{eqnarray}
The solution $\widehat{m}^*$ to the problem \eqref{obj3}-\eqref{norm3} is the $\|\cdot\|_2$-normalised eigenvector corresponding to the largest eigenvalue of the $\ell\times\ell$ matrix $U^\top U$ (see p.281, \cite{meyer2000matrix}).

In the particular case when $\{ \textbf{v}_1,\ldots,\textbf{v}_{\ell}\}$ is an orthonormal basis of $V$, we have that $E^\top E = \text{Id}_\ell$ and therefore
\[ \|\widehat{m}\|_2^2 = \boldsymbol\alpha^\top E^\top E\boldsymbol\alpha = \boldsymbol\alpha^\top \boldsymbol\alpha = \|\boldsymbol\alpha\|_2^2.\]
Using ~\eqref{Proj-on-Null}, we conclude that the optimization problem (\ref{obj3})-(\ref{norm3}) further simplifies to
\begin{eqnarray}
\label{obj4}
\max_{\boldsymbol\alpha\in\mathbb{R}^{\ell}} &&\|\widetilde{U}\boldsymbol\alpha\|_2^2\\
\mbox{subject to} &&\label{norm4}\|\boldsymbol\alpha\|_2^2-1=0,
\end{eqnarray}
where \begin{equation}\label{tildeU} \widetilde{U} = WE.\end{equation}
The solution $\boldsymbol\alpha^*$ to \eqref{obj4}-\eqref{norm4} is the eigenvector corresponding to the largest eigenvalue of $\widetilde{U}^\top \widetilde{U}$.
Finally, we note that the relationship between solutions of  \eqref{obj3}-\eqref{norm3} and \eqref{obj4}-\eqref{norm4} is given by \begin{equation}\label{malpha} \widehat{m}^* = E\boldsymbol\alpha^*.\end{equation}

\subsubsection{The optimal solution and optimal objective value}\label{optimal-value-posi-M}
For positive $M$, we can now derive an explicit expression for $E$ and thus obtain an explicit form for the solution of the optimization problem \eqref{obj}-\eqref{norm}.
We will do this by considering the reformulation \eqref{obj4}-\eqref{norm4} of our original problem \eqref{obj}-\eqref{norm}.
Let $V_0$ be the null space of $\textbf{1}^\top$. An orthonormal basis for $V_0$ is the set $\{\textbf{x}_1, \ldots , \textbf{x}_{n-1}\}$, where
\begin{equation}\label{Basis-ortho}
\textbf{x}_{i} = \frac{\widetilde{\textbf{x}}_i}{\|\tilde{\textbf{x}}_i \|_2}, \quad \text{for $i\in \{1, \ldots, n-1 \}$}
\end{equation}
and
\begin{equation}\label{orthog-vectors}
\widetilde{\textbf{x}}_1 = \left(\begin{array}{c}
1 \\
-1 \\
0 \\
\vdots \\
\vdots \\
0 \\
\end{array} \right),\ \widetilde{\textbf{x}}_2 = \left(\begin{array}{c}
1 \\
1 \\
-2 \\
0 \\
\vdots \\
0 \\
\end{array} \right), \dots, \widetilde{\textbf{x}}_{n-1} = \left(\begin{array}{c}
1 \\
\vdots \\
\vdots \\
\vdots \\
1 \\
-(n-1) \\
\end{array} \right).
\end{equation}
Let $B$ be an $n\times (n-1)$ matrix given by  \begin{equation}\label{B} B = (\textbf{x}_1|\dots |\textbf{x}_{n-1}).\end{equation}
Therefore, we can take  \begin{equation}\label{E1} E = \text{Id}_n\otimes B\end{equation} in~\eqref{E}.
Using Proposition \ref{Prop-Kron-Prod}(i), \eqref{lin-resp-formula-vectz} and~\eqref{tildeU},  we have $\widetilde{U}=WE = \textbf{h}^\top\otimes QB$. Hence, it follows from Proposition \ref{Prop-Kron-Prod}(i) and (iii) that
$$\widetilde{U}^\top \widetilde{U} = \textbf{h}\textbf{h}^\top\otimes B^\top Q^\top QB.$$
By Proposition \ref{Prop-Kron-Prod}(v), the eigenvector corresponding to the largest eigenvalue of $\widetilde{U}^\top \widetilde{U}$ is given by $\boldsymbol\alpha^* = \textbf{h}\otimes \textbf{y}$, where $\textbf{y}$ is the eigenvector corresponding
to the largest eigenvalue (which we denote by $\lambda$) of an  $(n-1)\times (n-1)$ matrix $B^\top Q^\top Q B$.
Hence, it follows from~\eqref{malpha} and~\eqref{E1} that \emph{the optimal perturbation is} \begin{equation}\label{hatm} \widehat{m}^* = E\boldsymbol\alpha^* = (\text{Id}_n\otimes B)(\textbf{h}\otimes \textbf{y}) = \textbf{h}\otimes B\textbf{y}.\end{equation}
Note that this expression for $\widehat{m}^*$ is an improvement over computing an eigenvector of the $(n^2-n)\times (n^2-n)$ matrix $\tilde{U}^\top\tilde{U}$ because we only need $\mathbf{y}$, an eigenvector if an $(n-1)\times (n-1)$ matrix.

Taking into account \eqref{norm3}, we must have  $\|\widehat{m}^* \|_2^2 =1$ and thus
$$1 = \widehat{m}^{*\top} \widehat{m}^* = (\textbf{h}^\top \textbf{h}) (\textbf{y}^\top B^\top B\textbf{y}) = \|\textbf{h}\|_2^2\cdot\|\textbf{y}\|_2^2,$$
as $B^\top B=\text{Id}_{n-1}$ (columns of $B$ form an orthonormal basis of $V_0$).
So, $\textbf{y}$ must satisfy \begin{equation}\label{normalization} \|\textbf{y}\|_2^2 = \frac{1}{\|\textbf{h}\|_2^2}.\end{equation}
Finally, using Proposition \ref{Prop-Kron-Prod}(ii), \eqref{lin-resp-formula-vectz} and~\eqref{hatm},
we obtain that  $$W\widehat{m}^* = (\textbf{h}^\top\otimes Q)(\textbf{h}\otimes B\textbf{y}) = \textbf{h}^\top \textbf{h} QB\textbf{y} = \|\textbf{h}\|_2^2 QB\textbf{y},$$
and therefore \emph{the optimal objective value is}
\begin{equation}\label{opt-obj-val-pos-M}
\|W\widehat{m}^*\|_2^2 = \|\textbf{h}\|_2^4\textbf{y}^\top B^\top Q^\top Q B \textbf{y} =\|\textbf{h}\|_2^4\textbf{y}^\top \left(\lambda \textbf{y}\right) =  \lambda\|\textbf{h}\|_2^4 \cdot \|\textbf{y}\|_2^2 = \lambda \|\textbf{h}\|_2^2.
\end{equation}
We impose the normalization condition~\eqref{normalization} for $\textbf{y}$ throughout the paper when dealing with positive $M$. Note that replacing $\widehat{m}$ with $-\widehat{m}^*$ in ~\eqref{opt-obj-val-pos-M} yields the same Euclidean norm of the response. We therefore choose the sign of $\widehat{m}^*$ so that $\|\mathbf{h}_M\|_2<\|\mathbf{h}_{M+\varepsilon m^*}\|_2$ for small $\varepsilon >0$.

In section \ref{Uniqueness-Result} we provide sufficent conditions for the optimal $m^*$ to be independent of the orthonormal basis vectors forming the columns of $B$ (or alternatively the columns of $E$). These conditions will also guarantee uniqueness of the optimal $m^*$ (up to sign).

\subsection{General transition matrix $M$ for mixing Markov chains}\label{Section-General-M}
In the  general setting, when $M$ is a transition matrix of an arbitrary mixing Markov chain, we consider the following optimization problem:
\begin{eqnarray}
\label{obj5}
\max_{m\in\R^{n\times n}} &&\|Qm\textbf{h}\|_2^2\\
\label{stoch2}\mbox{subject to} && m^\top\mathbf{1} =\mathbf{0}\\
\label{norm2}&&\|m\|_F^2-1=0\\
\label{zeros}&& m_{ij} = 0 \text{ if } M_{ij} = 0 \text{ or } 1.
\end{eqnarray}
Constraint \eqref{zeros} is imposed to ensure that if  it is impossible to transition from state $i$ to state $j$ in one step (i.e. $M_{ji} = 0$) or if state $i$ only leads to state $j$ (i.e. $M_{ji}=1$), then the perturbed Markov chain  will also have these properties.
We note that the solution to the optimization problem~\eqref{obj5}-\eqref{zeros} exists since we are again   maximising  a continuous function on  a compact subset  of $\R^{n\times n}$.

\subsubsection{Reformulating the problem (\ref{obj5})-(\ref{zeros}) in vectorized form}
As in the positive $M$ case, we want to find a matrix $A$ so that the constraints \eqref{stoch2} and \eqref{zeros} can be written in terms of $\widehat{m}$
in the linear form~\eqref{stoch-A}. Let
\begin{equation}\label{set} \mathcal{M}:= \{ i: \widehat{M}_i \in  \{0, 1\}\}=\{\gamma_1,\dots,\gamma_j\}\subset\{1,2,\dots,n^2\},\end{equation} where $\widehat{M}$ denotes the vectorization of $M$.
Proceeding as in the proof of Lemma~\ref{lem1}, it is easy to verify that constraints~\eqref{stoch2} and~\eqref{zeros} can be written in the form~\eqref{stoch-A}, where $A$ is a
$k\times n^2$ matrix ($k\geq n$) given by
\begin{equation}\label{A-final}
A = \left(\begin{array}{c}
 \text{Id}_{n}\otimes \mathbf{1}^\top \\
 \mathbf{e}_{\gamma_1}^\top \\
 \vdots \\
 \mathbf{e}_{\gamma_j}^\top
 \end{array} \right),
\end{equation}
where the $\textbf{e}_k$s in~\eqref{A-final}  are the $k$-th standard  basis vectors in $\R^{n^2}$.  As in the positive $M$ case,  the term $\text{Id}_n\otimes \textbf{1}^\top$ in~\eqref{A-final} corresponds to the constraint (\ref{stoch2}), while
all other entries of $A$ are related to constraints ~\eqref{zeros}. We conclude that
 we can  reformulate the optimization problem \eqref{obj5}-\eqref{zeros} in the form  \eqref{obj2}-\eqref{2-norm}  with $A$ given by \eqref{A-final}.

\subsubsection{Explicit construction of the orthonormal basis of the null space of the matrix $A$ in (\ref{A-final})}\label{Explicit-Basis}
Proceeding as in  the positive $M$ case, we want to  simplify the optimization problem (\ref{obj2})-(\ref{2-norm}) by constructing the matrix $E$ as in~\eqref{E},
whose columns form an  orthonormal basis for  the null space of $A$.
We first note that $E$ is   an $n^2\times\ell$ matrix, where $\ell$ is the nullity of $A$. Let us begin by computing $\ell$ explicitly.

\begin{lemma}\label{lemma-singularity}
The nullity of the matrix $A$ in (\ref{A-final}) is $n^2 - (n+n_1)$, where $n$ is the dimension of the square matrix $M$ and $n_1$ is the number of zero entries  in $M$.
\end{lemma}

\begin{proof}
Let \[ Y = \{\textbf{v}=(v_1, \ldots, v_n)\in\R^n:v_i = 1\text{ for some }1\leq i\leq n\}.\]
Assume first that  $M$ doesn't contain any columns that belong to $Y$ and  consider $\textbf{M}_j$, the $j$-th column of $M$.
Note that  the $j$-th row of $A$ is given by
\begin{equation}\label{Row-k-Atop}
(\underbrace{0,\dots,0}_{n(j-1)},\underbrace{1,\dots,1}_{n},\underbrace{0,\dots,0}_{n(n-j)}).
\end{equation}
On the other hand,  for every zero in $\textbf{M}_j$, we have the following row in $A$
\begin{equation}\label{Row-zero}
(\underbrace{0,\dots,0}_{n(j-1)},\underbrace{0,\dots,0,1,0,\dots,0}_{n},\underbrace{0,\dots,0}_{n(n-j)}),
\end{equation}
where $1$ is in a position corresponding to the position of the zero entry  in $\textbf{M}_j$.
Since $\textbf{M}_j \notin Y$, we have that the number of rows of the form~\eqref{Row-zero} in $A$ is at most $n-2$. Therefore,
we obviously have that the set spanned by row~\eqref{Row-k-Atop}
and rows~\eqref{Row-zero} is linearly independent. Moreover, since all other rows of $A$ have only zeros on places where vectors in~\eqref{Row-k-Atop} and~\eqref{Row-zero} have
nonzero entries and since $j$ was arbitrary, we conclude that rows of $A$ are linearly independent and that $\Rank(A)=n+n_1$. This immediately implies that the nullity of $A$ is $n^2 - (n+n_1)$.

The general case when $M$ can have columns  that belong to $Y$ can be treated similarly. Indeed, it is sufficient to note that
each $\textbf{M}_j \in Y$ will generate $n+1$ rows in $A$ (given again by
~\eqref{Row-k-Atop} and~\eqref{Row-zero}) but only form a subspace of dimension $n=1+(n-1)$ and $n-1$ is precisely the number of zero entries in $\textbf{M}_j$.
\end{proof}
For $A$ given by~\eqref{A-final} written in the form  \begin{equation}\label{ai} A = (A_1|\dots|A_n),  \quad \text{where $A_i\in\R^{k\times n}$,} \end{equation}
let $V$ be defined as in~\eqref{subspace-V}.
We will now construct the matrix $E$ as in~\eqref{E} whose columns form an orthonormal basis for $V$.
The first step is provided by the following result, where $\mbox{diag}(B_1,\dots,B_n)$ denotes the block matrix with diagonal blocks $B_1,\ldots,B_n$.
\begin{proposition}\label{Prop-Explicit-E}
The matrix $E$ has the form $E = \text{diag}(B_1,\dots,B_n)$,  where $B_i$ is the matrix whose columns form an   orthonormal basis of the null space of $A_i$ (if this null space is trivial,
we omit block $B_i$).
\end{proposition}

\begin{proof}
 Take an arbitrary   $\textbf{w}\in\R^{n^2}$ and write it in the form
$$\textbf{w} = \left(\begin{array}{c}
\textbf{w}_1 \\
\vdots \\
\textbf{w}_n
\end{array} \right), \quad \text{where $\textbf{w}_i\in\R^n$ for $1\le i \le n$.}$$
Noting that all entries of $A\textbf{w}$ are of the form $(A_i \textbf{w}_i)_j$ for $j$ such that $j$-th row of $A_i$ is nonzero, we conclude that
   $A\textbf{w}=\textbf{0}$ if and only if  $A_i\textbf{w}_i = \textbf{0}$ for each $i\in \{1,\ldots,n\}$. Moreover, each $\textbf{w}_i$ such that $A_i\textbf{w}_i = \textbf{0}$
 can be written as a linear combination of columns of $B_i$. Therefore, the columns of $E$ span the subspace $V$. The desired conclusion now follows from the simple observation that columns of $E$ form
 an orthonormal set.
\end{proof}
It remains to construct the matrices $B_i$, $1\le i \le n$, explicitly. Let us first introduce some additional notation.
For a matrix $J\in\R^{p_1\times p_2}$ and a set $L=\{l_1,\dots,l_s\}\subset \{1, \ldots, p_1\}$,
we define $J[L]$ to be  the  matrix consisting  of the rows $l_1,\ldots,l_s$ of  $J$. We note that $J[L]$ is an $s\times p_2$ matrix.

Let $A$ be given by~\eqref{A-final} and write it in the form~\eqref{ai}. Note that $A_i$ can be written as
\begin{equation}\label{eq}
A_i = \left( \begin{array}{c}
0_{i_1\times n} \\
\textbf{1}^\top \\
0_{i_2\times n} \\
\text{Id}_n[R_i] \\
0_{i_3\times n}
\end{array}  \right),
\end{equation}
where $R_i:= \{j: M_{ji} \in \{0, 1\} \}$ and
for some  $i_j\in \{0,1,\dots, n^2\}$, $j\in \{1, 2, 3\}$ such that  \[ \sum_{j=1}^3 i_j = k-|R_i|-1;\]
recall $A$ has $k$ rows (see (\ref{A-final})).
It follows from~\eqref{eq} that the null space of $A_i$ is the same as  the null space of the matrix $$ \widetilde{A}_i :=\left( \begin{array}{c}
 \textbf{1}^\top \\
\text{Id}_n[R_i]
\end{array}\right).$$
Let $r_i\in \{0, \ldots, n-1\}$ denote  the number of zeros in the $i$-th column of $M$.  It follows from the arguments in the proof of Lemma~\ref{lemma-singularity} that
\begin{equation}\label{rank}
\Rank (\widetilde{A}_i)= r_i+1.
\end{equation}
In particular, when $r_i=n-1$, the nullity of $\widetilde{A}_i$ is zero.
\begin{proposition}\label{Prop-Explicit-B} Assume that $r_i<n-1$ and let \[\widetilde{B}_i = (\mathbf{x}_1|\dots|\mathbf{x}_{(n-1)-r_i})\in\R^{(n-r_i) \times ((n-1)-r_i)},\]
where $\mathbf{x}_i$ are given by~\eqref{Basis-ortho}. Furthermore,  let $B_i\in\R^{n \times ((n-1)-r_i)}$ be a matrix defined by the conditions:
\begin{equation}\label{Bi} B_i[R_i] = 0_{r_i\times ((n-1)-r_i)} \quad  \text{and} \quad B_i[\{1,\dots,n\}\setminus R_i] = \widetilde{B}_i.\end{equation}
Then, the columns of $B_i$ form an  orthonormal basis  for  the null space of $A_i$.
\end{proposition}
\begin{proof}
As the null spaces of matrices $A_i$ and $\widetilde{A}_i$ coincide, it is sufficient to show that columns of $B_i$ form an orthonormal basis for the null space of $\widetilde{A}_i$.
We first note that the orthonormality of $\mathbf{x}_1, \ldots,\mathbf{x}_{n-1-r_i}$ in $\R^{n-r_i}$ directly implies that the columns of $B_i$ form an orthonormal set in $\R^n$, since the $j$-th
column of $B_i$ is built from $\mathbf{x}_j$ by adding zeroes on appropriate places that are independent of $j$. Furthermore, since $\mathbf{x}_1, \ldots, \mathbf{x}_{n-1-r_i}$ are in the null space
of $\textbf{1}^\top_{n-r_i}$, we have that the columns of $B_i$ belong to the null space of $\textbf{1}^\top_{n}$. Moreover, it follows from the first equality in~\eqref{Bi} that
columns of $B_i$ are also orthogonal to  all other rows of $\widetilde{A}_i$. Consequently, we conclude all columns of $B_i$ lie in the null space of $\widetilde{A}_i$.
Finally, by~\eqref{rank} we have that the nullity of $\widetilde{A}_i$ is $n-r_i-1$ which is the same as the  number of columns of $B_i$ and therefore columns of $B_i$ span the null space of $\widetilde{A}_i$.

\end{proof}

\subsubsection{Solution to the problem (\ref{obj5})-(\ref{zeros})}\label{sect-gen-final-soln}
Now that we have constructed an appropriate $E$ (Proposition \ref{Prop-Explicit-E} gives the form of $E$ and Proposition \ref{Prop-Explicit-B} provides the specific components of $E$),
we can reformulate our problem \eqref{obj2}-\eqref{2-norm} (with the matrix $A$ in \eqref{A-final}),
to obtain the optimization problem \eqref{obj4}-\eqref{norm4} with $\widetilde{U}$ as in~\eqref{tildeU}.
The vectorized solution to \eqref{obj5}-\eqref{zeros} is given by $\widehat{m}^*$ as in~\eqref{malpha},
where $\boldsymbol\alpha^*$ again denotes  the eigenvector corresponding to the largest eigenvalue of the matrix $\widetilde{U}^\top\widetilde{U}$.
Finally, as for the positive $M$ case, we have that both $m^*$ and $-m^*$ yield the same Euclidean norm of the response ~\eqref{obj5}. Hence, we choose the sign of the matrix $m^*$ so that $\|\mathbf{h}_M\|_2<\|\mathbf{h}_{M+\varepsilon m^*}\|_2$ for small $\varepsilon>0$.

\subsubsection{A sufficient condition for a unique optimal solution and independence of the choice of basis of the null space of $A$}\label{Uniqueness-Result}
The following result provides an easily checkable sufficient condition for the uniqueness of the solution $m^*$ (up to sign) to the problems (\ref{obj})-(\ref{norm}) and (\ref{obj5})-(\ref{zeros}). Under this condition, the specific choice of basis for the null space of the constraint matrix $A$ is unimportant, and the $m^*$ computed in Algorithms 1 and 2 in section \ref{sect:Computations} is independent of this basis choice.

\begin{proposition}\label{Prop-Unique}
Suppose that $\widetilde{U}_1 = WE_1$ and $\widetilde{U}_2 = WE_2$, with $E_1\neq E_2$,  and that the columns of $E_1$ and $E_2$ each form an orthonormal basis for the null space of $A$. Let $\boldsymbol\alpha^*_1$ and $\boldsymbol\alpha^*_2$ be the eigenvectors corresponding to the largest eigenvalues $\lambda_1$ and $\lambda_2$ of $\widetilde{U}_1^\top\widetilde{U}_1$ and $\widetilde{U}_2^\top\widetilde{U}_2$, respectively, normalised so that $\|\boldsymbol\alpha^*_1\|_2=\|\boldsymbol\alpha^*_2\|_2=1$. If $\lambda_1$ has multiplicity one then $\lambda_2$ also has multiplicity one and $\widehat{m}^*_1 = E_1 \boldsymbol\alpha^*_1 = \theta E_2 \boldsymbol\alpha^*_2 = \theta\widehat{m}^*_2$, where $\theta\in\{-1,1\}$.
\end{proposition}

\begin{proof}
Let $E_1, E_2\in\R^{n^2\times\ell}$ be the matrices with columns consisting of orthonormal basis vectors of the null space of $A$ such that $E_1\neq E_2$. As both $E_1$ and $E_2$ span the same space, there exists some matrix $R\in\R^{\ell\times\ell}$ such that $E_2=E_1R$. Noting that $E_i^\top E_i =$ Id$_{\ell}$, $i=1,2$, we have that Id$_{\ell} = E_2^\top E_2 = R^\top E_1^\top E_1 R = R^\top R$; using the fact that $R$ is square, we also have that $R^\top = R^{-1}$ and hence $R$ is orthogonal. As
\begin{equation*}
\widetilde{U}_1^\top\widetilde{U}_1 = E_1^\top W^\top W E_1\text{ and }\widetilde{U}_2^\top\widetilde{U}_2 = R^{-1} E_1^\top W^\top W E_1 R,
\end{equation*}
the matrices $\widetilde{U}_1^\top\widetilde{U}_1$ and $\widetilde{U}_2^\top\widetilde{U}_2$ are similar; thus, $\lambda_2$ has multiplicity one. Combining this with the fact that $\|\boldsymbol\alpha^*_1\|_2=\|\boldsymbol\alpha^*_2\|_2=1$, we finally have that $\boldsymbol\alpha^*_1 = \theta R\boldsymbol\alpha^*_2$  and $$\widehat{m}^*_1 = E_1 \boldsymbol\alpha^*_1 = \theta E_1 R \boldsymbol\alpha^*_2 = \theta E_2 \boldsymbol\alpha^*_2 = \theta\widehat{m}_2^*.$$
\end{proof}

\subsection{Computations}\label{sect:Computations}
In this section, we will discuss computational aspects of the content presented so far.

\subsubsection{Computing the response vector $\mathbf{u}_1$ without forming $Q$}
\label{Without-Q}
So far, we  have used (\ref{Lin-Resp}) to represent $\mathbf{u}_1$, which requires computation of $Q$, which itself requires inversion of a possibly large matrix.
We note that we can avoid computing $Q$ explicitly. Indeed, we can find the  linear response $\textbf{u}_1$ as a  unique solution for the  linear system:
\begin{equation*}
\widetilde{M}\textbf{u}_1 = \boldsymbol\kappa,
\end{equation*}
where
\begin{equation}\label{tildeM} \widetilde{M} = \left(\begin{array}{c}
\text{Id}-M \\
\textbf{1}^\top
\end{array}\right)\text{ and }
\boldsymbol\kappa = \left(\begin{array}{c}
m\textbf{h} \\
0
\end{array}\right).\end{equation}
This approach is more numerically stable than directly forming $Q$ by matrix inversion and exploits sparseness of $M$ when present.

\subsubsection{Computing the optimal perturbation $\hat{m}^*$ for positive $M$ without forming $Q$}\label{Pos-M-No-Q}
In section \ref{optimal-value-posi-M}, under the assumption that $M$ is positive,  we showed that the solution to our optimization problem is given by~\eqref{hatm},
 where $\textbf{y}$ is the eigenvector corresponding to the largest eigenvalue of $B^\top Q^\top Q B$ and $B$ is given by~\eqref{B}. We claim that we can find $QB$ without having
 to explicitly compute $Q$.

 Let us first note  that $\textbf{1}^\top(\text{Id} - M +\textbf{h}\textbf{1}^\top) = \textbf{1}^\top$ and thus $\textbf{1}^\top Q = \textbf{1}^\top$. Hence,
 for $\textbf{w}\in\R^n$ such that $\textbf{1}^\top\textbf{w} = 0$, we have that $\textbf{1}^\top Q\textbf{w} = 0$.
 Using this and the fact that each column of $B$ sums to zero, we can can find $QB$ as a solution to the linear system
\begin{equation}\label{X}
\widetilde{M}X = K,
\end{equation}
where
\begin{equation}\label{k}
K = \left(\begin{array}{c}
B \\
\textbf{0}^\top
\end{array}\right)
\end{equation}
and $\widetilde M$ as in~\eqref{tildeM}.
Thus, replacing $B^\top Q^\top QB$ with $X^\top X$, we avoid matrix inversion to form $Q$ and do not deal with matrices of order larger then  $(n+1)\times n$.

We note that it is also possible to compute the matrices $U$ and $\widetilde{U}$ in equations ~\eqref{eqU} and \eqref{tildeU} for a general transition matrix $M$ of a mixing Markov chain without forming $Q$. However, doing this might not be much more efficient than computing $Q$.

\subsubsection{Algorithms for solving (\ref{obj})-(\ref{norm}) and (\ref{obj5})-(\ref{zeros})}\label{Algo}
We first present the algorithm for finding the solution $m^*$ of the problem (\ref{obj})-(\ref{norm}).
\begin{multicols}{2}
\textbf{Algorithm 1}
\begin{enumerate}
\item Compute $\textbf{h}$ as the invariant probability vector of $M$.
\item Construct the matrix $B$ in ~\eqref{B}.
\item Solve the linear equation ~\eqref{X} for the matrix $X$.
\item Compute the singular vector $\mathbf{y}$ corresponding to the largest singular value of $X$.
\item Form the matrix $m^*$ using ~\eqref{hatm} and the normalisation ~\eqref{normalization}.
\end{enumerate}
\columnbreak
\textbf{Matlab Code}
\begin{verbatim}
function [m,h] = lin_resp(M)
n=length(M);

%Step 1
[V,D] = eigs(M,1);
h = V;
h = h/sum(h);

%Step 2
B = triu(ones(n))-diag([1:n-1],-1);
B(:,n) = [];
B = sparse(normc(B));

%Step 3
X = [speye(n)-M;ones(1,n)]\[B;zeros(1,n-1)];

%Step 4
[U2,D2,V2] = svds(X,1);

%Step 5
y = 1/(norm(h)*norm(V2))*V2;
m = B*y*h';
end
\end{verbatim}
\end{multicols}
Note that finally, one needs to select the correct sign of $m$. Next, we state the algorithm for solving (\ref{obj5})-(\ref{zeros}).
\newpage
\begin{multicols}{2}
\textbf{Algorithm 2}
\begin{enumerate}
\item Compute $\textbf{h}$ as the invariant probability vector of $M$.
\item Construct the matrix $B$ in ~\eqref{B} .
\item Compute the matrix $\widetilde{U}$ using results from Propositions \ref{Prop-Explicit-E} and \ref{Prop-Explicit-B}, Lemma \ref{lemma-singularity} and the fact that $\widetilde{U} = Q(\mathbf{h}^\top\otimes \text{Id}_n)E$; where this last expression is derived using Proposition \ref{Prop-Kron-Prod}(i).
\item Compute the singular vector $\boldsymbol\alpha^*$ corresponding to the largest singular value of $\widetilde{U}$.
\item Form the matrix $m^*$ by using the results from Propositions \ref{Prop-Explicit-E} and \ref{Prop-Explicit-B} and the fact that $\widehat{m}^* = E\boldsymbol\alpha^*$.
\end{enumerate}
\columnbreak
\textbf{Matlab Code}
\begin{verbatim}
function [m,h] = lin_resp(M)
n=length(M);

%Step 1
[V,D] = eigs(M,1);
h = V;
h = h/sum(h);

%Step 2
B = triu(ones(n))-diag([1:n-1],-1);
B(:,n) = [];
B = sparse(normc(B));

%Step 3
n1 = length(find(M==0));
U = zeros(n,n^2-(n+n1));
j1 = 1;
j2 = 0;
for i=1:n
    R = find(M(:,i)==0);
    r = length(R);
    if r~= n-1
        B_i = zeros(n,n-r-1);
        R2 = setdiff([1:n],R);
        r2 = length(R2);
        B_i(R2,:) = B(1:r2,1:(r2-1));
        j2 = j2+n-r-1;
        U(:,j1:j2) = h(i)*B_i;
        j1 = j2+1;
    end
end

M_inf = h*ones(1,n);
Q = inv(eye(n)-M+M_inf);
U = Q*U;

%Step 4
[U2,D2,V2] = svds(U,1);

%Step 5
m = sparse(n,n);
j1=1;
j2=0;
for i=1:n
    R = find(M(:,i)==0);
    r = length(R);
    j2=n-r-1+j2;
    if r~= n-1
        B_i = zeros(n,n-r-1);
        R2 = setdiff([1:n],R);
        r2 = length(R2);
        B_i(R2,:) = B(1:r2,1:(r2-1));
        m(:,i) = B_i*V2(j1:j2);
    else
        m(:,i) = sparse(n,1);
    end
    j1=j2+1;
end
end
\end{verbatim}
\end{multicols}
Note that again we must select the correct sign of $m$.
\subsection{Analytic examples}
\subsubsection{Analytic Solution for $M\in\R^{2\times 2}$}\label{example-analytic-2x2}
We will now construct explicitly the solution for the  problem \eqref{obj5}-\eqref{zeros} when $M\in\R^{2\times 2}$. Since $M$ is column stochastic and since columns of $m$ sum to zero, we can write
\begin{equation*}
M = \left(\begin{array}{cc}
M_{11} & M_{12} \\
M_{21} & M_{22}
\end{array}\right) =  \left(\begin{array}{cc}
1-M_{21} & M_{12} \\
M_{21} & 1-M_{12}
\end{array}\right)
\end{equation*}
and
\[
m = \left(\begin{array}{cc}
m_{11} & m_{12} \\
m_{21} & m_{22}
\end{array}\right) =  \left(\begin{array}{cc}
m_{11} & -m_{22} \\
-m_{11} & m_{22}
\end{array}\right).
\]
Furthermore, let  $\textbf{h} = (h_1, h_2)^\top$. We first note that without any loss of generality, we can assume that $M$ is positive.
Indeed, if  $M_{11} = 0$ then by~\eqref{norm2} and~\eqref{zeros}, we have that $m_{11}=0$ and  $m_{22}=\pm \frac{1}{\sqrt{2}}$. Similarly, if $M_{22} = 0$ then $m_{22}=0$ and $m_{11}=\pm \frac{1}{\sqrt{2}}$
Furthermore, we note that $M_{11} \neq 1$ and $M_{22} \neq 1$ since otherwise $M$ would not be a transition matrix of an ergodic
Markov chain.

We therefore assume that $M$ is positive.  We begin by noting that the invariant probability vector for $M$ is given by
\begin{equation}
\label{h2x2eqn}
\textbf{h} = d\left(\begin{array}{c}
M_{12}\\
M_{21}
\end{array}\right),
\end{equation}
where $d = \frac{1}{M_{12}+M_{21}}$.
It follows from~\eqref{hatm} that $\widehat{m}^* = \textbf{h}\otimes B\textbf{y}$ and thus $m^* = B\textbf{y}\textbf{h}^\top$,
where $B$ is given by~\eqref{B} and $\textbf{y}$ is the eigenvector corresponding to the largest eigenvalue of $B^\top Q^\top QB$.
 Observe that in this case
\begin{equation*}
B = \left(\begin{array}{c}
\frac{1}{\sqrt{2}}\\
-\frac{1}{\sqrt{2}}
\end{array}\right).
\end{equation*}
In order to find $\textbf{y}$, we begin by computing $QB$. In section~\ref{Pos-M-No-Q}  we have observed that $QB$ is given by $X$, where $X$ solves
$\widetilde M X=K$ with  $\widetilde M$ and $K$   given by~\eqref{tildeM} and~\eqref{k} respectively. Hence, $X$ solves the system
\begin{equation*}
\left(\begin{array}{cc}
M_{21} & -M_{12} \\
-M_{21} & M_{12} \\
1 & 1
\end{array}  \right)X =
\left(\begin{array}{c}
\frac{1}{\sqrt{2}} \\
-\frac{1}{\sqrt{2}} \\
0
\end{array} \right),
\end{equation*}
and therefore
\begin{equation*}
QB = X = \frac{1}{M_{12}+M_{21}} \left(\begin{array}{c}
\frac{1}{\sqrt{2}} \\
-\frac{1}{\sqrt{2}}
\end{array} \right)=dB.
\end{equation*}
Consequently,  \[ (QB)^\top QB = d^2 B^\top B = d^2, \]
and therefore  $\textbf{y}$ in this case is a scalar.
Taking into account~\eqref{normalization}, we can take  \[\textbf{y} = \pm\frac{ 1}{\|\textbf{h}\|_2}=\pm\frac{1}{d\sqrt{M_{12}^2+M_{21}^2}}, \] which yields
\begin{equation}\label{opt}
m^* = B\textbf{y}\textbf{h}^\top =\pm \frac{1}{\sqrt{2(M_{12}^2+M_{21}^2)}} \left(\begin{array}{cc}
M_{12} & M_{21} \\
-M_{12} &-M_{21}
\end{array} \right).
\end{equation}
At this point, we select the sign of $m^*$ to ensure that $\|\mathbf{h}_{M+\varepsilon m^*}\|_2>\|\mathbf{h}_M\|_2$ for small $\varepsilon>0$.
Inspecting (\ref{h2x2eqn}) we see that if $M_{12}>M_{21}$ we should increase $M_{12}$ and decrease $M_{21}$ to increase $\|\mathbf{h}\|_2$.
Thus,
\begin{eqnarray}
\label{optcase}
m^* = \left\{
        \begin{array}{ll}
          \frac{1}{\sqrt{2(M_{12}^2+M_{21}^2)}} \left(\begin{array}{cc}
M_{12} & M_{21} \\
-M_{12} &-M_{21}
\end{array} \right), & \hbox{if $M_{12}\ge M_{21}$;} \\
          \frac{1}{\sqrt{2(M_{12}^2+M_{21}^2)}} \left(\begin{array}{cc}
-M_{12} & -M_{21} \\
M_{12} &M_{21}
\end{array} \right), & \hbox{if $M_{21}>M_{12}$.}
        \end{array}
      \right.
\end{eqnarray}

Finally, it follows from~\eqref{Lin-Resp} that
\begin{eqnarray*}
\lefteqn{\mathbf{u}_1=Qm^*\textbf{h} = \|\textbf{h}\|_2^2QB\textbf{y} = \|\textbf{h}\|_2^2 dB\frac{1}{\|\textbf{h}\|_2} = \|\textbf{h}\|_2dB}\\
 &=& \left\{
        \begin{array}{ll}
           \frac{\sqrt{M_{12}^2+M_{21}^2}}{(M_{12}+M_{21})^2}\left(\begin{array}{c}
\frac{1}{\sqrt{2}} \\
-\frac{1}{\sqrt{2}}
\end{array} \right) , & \hbox{if $M_{12}\ge M_{21}$;} \\
           \frac{\sqrt{M_{12}^2+M_{21}^2}}{(M_{12}+M_{21})^2}\left(\begin{array}{c}
-\frac{1}{\sqrt{2}} \\
\frac{1}{\sqrt{2}}
\end{array} \right) , & \hbox{if $M_{21}>M_{12}$.}
        \end{array}
      \right.
\end{eqnarray*}
and thus
\begin{equation*}
\|\mathbf{u}_1\|_2^2=\|Qm^*\textbf{h}\|_2^2 = k^2\|\textbf{h}\|_2^2= \frac{M_{12}^2+M_{21}^2}{(M_{12}+M_{21})^4}.
\end{equation*}
The minimum value of this expression occurs when $M_{12}=M_{21}=1$ (value of $1/8$) and increases with decreasing values of $M_{12}$ and $M_{21}$.
There is a singularity at $M_{12}=M_{21}=0$ when the second eigenvalue merges with the eigenvalue 1;  see Figure \ref{fig:contour2x2}.
\begin{figure}
  \centering
  \includegraphics[width=10cm]{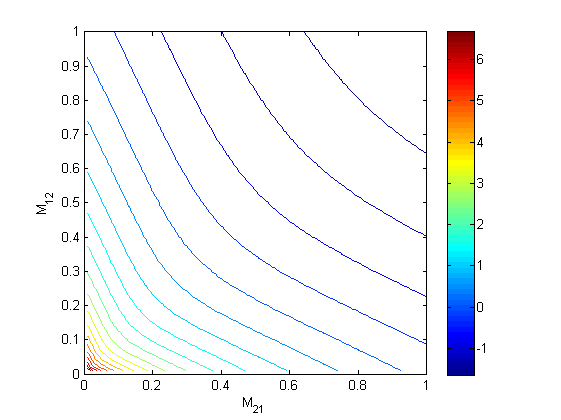}
  \caption{Contour plot of $\log_e((M_{12}^2+M_{21}^2)/(M_{12}+M_{21})^4)$}\label{fig:contour2x2}
\end{figure}

\subsubsection{Analytic Solution for $M = \frac{1}{n} 1_{n\times n}$}
Let us now solve explicitly the problem \eqref{obj5}-\eqref{zeros} when  $M=\frac{1}{n}1_{n\times n}$.
Note  that $\textbf{h} = \frac{1}{n}\textbf{1}_n$.
As in the previous section, it follows from~\eqref{hatm} that $m^* = B\textbf{y}\textbf{h}^\top$, where $B$ is given by~\eqref{B}
and $\textbf{y}$ is the eigenvector corresponding to the largest eigenvalue of $B^\top Q^\top QB$.
Observe that
\begin{equation*}
\begin{aligned}
Q &= (\text{Id}_n-M+\textbf{h}\textbf{1}_n^\top)^{-1}\\
&= \left(\text{Id}_n-\frac{1}{n}1_{n\times n}+\frac{1}{n}\textbf{1}_n\textbf{1}_n^\top\right)^{-1}\\
&= \text{Id}_n^{-1} = \text{Id}_n,
\end{aligned}
\end{equation*}
and thus $QB=B$.
Hence, we have that $B^\top Q^\top QB = B^\top B = \text{Id}_{n-1}$.
Taking into account~\eqref{normalization}, we observe that  \[ \|\textbf{y}\|^2_2 = \frac{1}{\|\textbf{h}\|^2_2}=n.\]
Therefore, we  can take  $\textbf{y} = \sqrt{\frac{n}{n-1}}\textbf{1}_{n-1}$.
Consequently,
\begin{equation*}
m^* = B\textbf{y}\textbf{h}^\top = \frac{1}{\sqrt{n(n-1)}}B1_{(n-1)\times n}
=\frac{1}{\sqrt{n(n-1)}}\sum_{i=1}^{n-1}\mathbf{x}_i\otimes\mathbf{1}^\top_n
\end{equation*}
which yields
\begin{equation*}
\mathbf{u}_1=Qm^*\textbf{h} = \frac{1}{\sqrt{n(n-1)}}\sum_{i=1}^{n-1} \left(\textbf{x}_i \otimes\mathbf{1}^\top_n\right)\left(\frac{1}{n}\textbf{1}_n\right) = \frac{1}{\sqrt{n(n-1)}}\sum_{i=1}^{n-1} \textbf{x}_i.
\end{equation*}
and
\begin{equation}\label{2norm}
\|\mathbf{u}_1\|_2^2=\|Qm^*\textbf{h}\|_2^2 = \frac{1}{{n(n-1)}}\sum_{i=1}^{n-1} \mathbf{x}_i^\top \mathbf{x}_i = 1/n.
\end{equation}
In this example, the sign of $m^*$ is unimportant because perturbations of $M$ by both $m^*$ and $-m^*$ result in the same increase of $\|\mathbf{h}\|_2$.
We note that the optimal $m^*$ is not unique since it depends on orthonormal basis for the null-space of $\textbf{1}^\top$.
The Euclidean norm of the linear response of the invariant probability vector $\mathbf{h}$ is independent of basis choice.

\section{Maximising the linear response of the expectation of an observable}
\label{sect:expectation}
In this section, we consider maximizing the linear response of the expected value of a cost vector $\mathbf{c}$ with respect to the invariant probability vector $\mathbf{h}$. The computations developed in this section will be used in section \ref{sect:numeric} to solve a discrete version of the problem of maximizing the linear response of an observable with respect to the invariant measure of a stochastically perturbed dynamical system.

We recall that the linear response to the invariant probability vector $\mathbf{h}$ of a transition matrix $M$ under a perturbation matrix $m$ is denoted by $\mathbf{u}_1$.
Therefore we wish to select a perturbation matrix $m$ so that we maximise $\mathbf{c}^T\mathbf{u}_1$.
For $\mathbf{c}\in\R^n$, using (\ref{Lin-Resp}), we consider the following problem:
\begin{eqnarray}
\label{obj_lin_fun}
\max_{m\in\R^{n\times n}} &&\mathbf{c}^\top Qm\textbf{h}\\
\label{stoch_lin_fun}\mbox{subject to} && m^\top\mathbf{1} =\mathbf{0}\\
\label{norm_lin_fun}&&\|m\|_F^2-1=0\\
\label{zero_lin_fun}&& m_{ij} = 0 \text{ if } (i,j)\in N,
\end{eqnarray}
where $N = \{(i,j)\in \{1,\ldots,n\}^2: M_{ij}= 0 \text{ or } 1\}$. Note that as $m_{ij}$ takes the value 0 for all $(i,j)\in N$, we just need to solve (\ref{obj_lin_fun})--(\ref{norm_lin_fun}) for $(i,j)\not\in N$.

We employ Lagrange multipliers.
Consider the Lagrangian function
\begin{equation}\label{Lag-fun-lin-fun}
L(m,\boldsymbol\varrho,\nu) = \textbf{w}^\top m\textbf{h}-\boldsymbol\varrho^\top m^\top \textbf{1} - \nu(\|m\|_F^2-1),
\end{equation}
where $\textbf{w}^\top = \mathbf{c}^\top Q \in\R^n$ and $\boldsymbol\varrho\in\R^n, \nu\in\R$ are the Lagrange multipliers. Differentiating (\ref{Lag-fun-lin-fun}) with respect to $m_{ij}$, we obtain
\begin{equation*}
\frac{\partial L}{\partial m_{ij}}(m,\boldsymbol\varrho,\nu) = w_ih_j-\varrho_j-2\nu m_{ij}.
\end{equation*}

Using the method of Lagrangian multipliers, we require
\begin{equation}
\label{eq72}
w_ih_j-\varrho_j-2\nu m_{ij} = 0\text{ for }(i,j)\not\in N
\end{equation}
and
\begin{equation}\label{stoch_lin_fun_2}
\sum_{i:(i,j)\not\in N}m_{ij}=0\text{ for }j\in\{1,\dots,n\}.
\end{equation}
Equation (\ref{eq72}) yields $\varrho_j = -2\nu m_{ij} +w_ih_j$ for $(i,j)\not\in N$.
Using (\ref{stoch_lin_fun_2}), we calculate
\begin{equation*}
\sum_{i:(i,j)\not\in N}\varrho_j = |N_j^c|\varrho_j = h_j\sum_{i:(i,j)\not\in N}w_i,
\end{equation*}
where $N_j^c = \{i: (i,j)\not\in N\}$.
Thus, substituting $\varrho_j=(h_j/|N_j^c|)\sum_{l:(l,j)\not\in N}w_l$ we obtain
\begin{equation}
\label{60a}
m_{ij}^* = \frac{-\varrho_j+w_ih_j}{2\nu} = \frac{h_j}{2\nu}\left(w_i-\frac{1}{|N_j^c|}\sum_{l:(l,j)\not\in N}w_l\right),
\end{equation}
We determine the sign of $\nu$ by checking the standard sufficient second order conditions for $m_{ij}^*$ to be a maximum (see e.g.\ Theorem 9.3.2 \cite{fletcher1991practical}).
The matrix $m^*$ satisfies the first-order equality constraints (\ref{stoch_lin_fun})-(\ref{zero_lin_fun}) and $\frac{\partial L}{\partial m_{ij}}(m^*,\boldsymbol\varrho,\nu)=0$ for $(i,j)\not\in N$.
We compute
\begin{equation}\label{Hessian-expected}
\frac{\partial^2 L}{\partial m_{ij}\partial m_{kl}}(m^*,\boldsymbol\varrho,\nu) = -2\nu\delta_{(i,j),(k,l)};
\end{equation}
thus, the Hessian matrix of the Lagrangian function is $H(m^*,\boldsymbol\varrho,\nu) = -2\nu$ Id$_{n^2-|N|}$.
If $\nu>0$ then for any $\mathbf{s}\in\R^{n^2-|N|}\setminus\{\mathbf{0}\}$ (indeed for any $\mathbf{s}\in \R^{n^2}\setminus \{\mathbf{0}\}$), one has $\mathbf{s}^\top H(m^*,\boldsymbol\varrho,\nu) \mathbf{s} <0$, satisfying the second-order sufficient condition.

Using ~\eqref{60a}, we must select $\nu$ to ensure $\|m^*\|_F=1$. Writing $m^*_{ij}=\frac{\widetilde{m}_{ij}}{\nu}$, the constraint $\|m^*\|_F=1$ implies $\nu =\theta\|\widetilde{m}\|_F$, where $\theta\in\{-1,1\}$.
As we require that $\nu>0$ for the solution to be the maximiser, we conclude that $\theta=1$ and
\begin{equation}\label{nu-Expectedvalue}
\nu = \|\widetilde{m}\|_F.
\end{equation}

\subsection{Algorithm for solving problem (\ref{obj_lin_fun})-(\ref{zero_lin_fun})}\label{Algo-fun}
We can solve problem (\ref{obj_lin_fun})-(\ref{zero_lin_fun}) using the following algorithm, which exploits sparsity of $M$.
\newpage
\begin{multicols}{2}
\textbf{Algorithm 3}
\begin{enumerate}
\item Compute the invariant probability vector $\textbf{h}$ of $M$.
\item Solve $\left(I_n-M+\textbf{h}\textbf{1}^\top\right)^\top\textbf{w} = \textbf{c}$ for $\textbf{w}$.
\item Calculate $m^*_{ij}$ according to (\ref{60a}), where $\nu$ is given by ~\eqref{nu-Expectedvalue}.

\end{enumerate}
\columnbreak
\textbf{Matlab Code}
\begin{verbatim}
function m = lin_resp_fun(M,c)
n=length(M);
%Step 1
[V,D] = eigs(M,1);
h = V;
h = h/sum(h);
%Step 2
Z = eye(n)-M+h*ones(1,n);
w = Z'\c;
%Step 3
m = zeros(n);
for j=1:n
    N_j = find(M(:,j)>10^-7);
    if(length(N_j) > 1)
        m(N_j,j) = h(j)*(w(N_j)- mean(w(N_j)));
    end
end
m = m./(norm(m,'fro'));
end
\end{verbatim}
\end{multicols}
\begin{remark}
For large, sparse $M$, one can replace Step 2 above with:  Solve the following (sparse) linear system for $w$
\begin{equation}
\left(\begin{array}{c}
I-M^\top \\
\mathbf{h}^\top
\end{array} \right)\mathbf{w} = \left( \begin{array}{c}
\mathbf{c}-(\mathbf{h}^\top\mathbf{c})\mathbf{1} \\
\mathbf{h}^\top\mathbf{c}
\end{array} \right).
\end{equation}
\end{remark}
\subsection{Analytic examples}
\subsubsection{Analytic Solution for $M\in \mathbb{R}^{2\times 2}$}
Suppose that $M\in\R^{2\times 2}$ and we would like to solve (\ref{obj_lin_fun})-(\ref{zero_lin_fun}) for $\mathbf{c}\in\R^2$, $\mathbf{c}\neq a\textbf{1}$, where $a\in\R$. As in the example in section \ref{example-analytic-2x2}, we only need to consider the case when $M$ is positive. From section \ref{example-analytic-2x2}, we have that
$$ M =  \left(\begin{array}{cc}
1-M_{21} & M_{12} \\
M_{21} & 1-M_{12}
\end{array}\right)
\text{ and }\mathbf{h} = d \left(\begin{array}{c}
M_{12} \\
M_{21}
\end{array}  \right),$$
where $d=\frac{1}{M_{12}+M_{21}}$. From ~\eqref{60a}, the solution is given by
\begin{equation*}
\begin{aligned}
m^*_{11} = -m^*_{21} = \frac{1}{2\nu}\frac{h_1}{2}(w_1-w_2) \\
m^*_{22} = -m^*_{12} = \frac{1}{2\nu}\frac{h_2}{2}(w_2-w_1) ,\\
\end{aligned}
\end{equation*}
where $\mathbf{w} = Q^\top\mathbf{c}$. As $\mathbf{c}\neq a\textbf{1}$, we have that $w_1-w_2\neq 0$; this is the case since using the equation $\mathbf{w} = Q^\top\mathbf{c}$ and the fact that $\mathbf{1}^\top$ is a left eigenvector of $Q^{-1}$, we see that if $\mathbf{w} = a\mathbf{1}$ then $\mathbf{c} = a\textbf{1}$, which is a contradiction. The constraint $\|m^*\|^2_F = 1$ implies $(2\nu)^2= \frac{1}{2}\left((h_1^2+h_2^2)(w_1-w_2)^2\right) = d^2(w_1-w_2)^2\left(\frac{M_{12}^2+M_{21}^2}{2}\right)$ and so $2\nu = \theta d(w_1-w_2)\sqrt{\frac{M_{12}^2+M_{21}^2}{2}}$, where $\theta\in\{-1,1\}$. Therefore
\begin{equation}\label{m-opt-exp-2x2}
\begin{aligned}
m^* &= \frac{\theta}{d(w_1-w_2)\sqrt{2(M_{12}^2+M_{21}^2)}} \left(\begin{array}{cc}
d M_{12} (w_1-w_2) & dM_{21} (w_1-w_2)  \\
dM_{12} (w_2-w_1)  &dM_{21} (w_2-w_1)
\end{array} \right)\\
& = \frac{\theta}{\sqrt{2(M_{12}^2+M_{21}^2)}} \left(\begin{array}{cc}
M_{12} & M_{21} \\
-M_{12} &-M_{21}
\end{array} \right).
\end{aligned}
\end{equation}
As we require $\nu >0$ (see discussion following equation ~\eqref{Hessian-expected} for the maximisation condition), we will need that sign($\theta (w_1-w_2))>0$. Thus, if $w_1-w_2>0$ then $\theta=1$ and if $w_1-w_2<0$ then $\theta=-1$. Using this we obtain
\begin{eqnarray}
m^* = \left\{
        \begin{array}{ll}
          \frac{1}{\sqrt{2(M_{12}^2+M_{21}^2)}} \left(\begin{array}{cc}
M_{12} & M_{21} \\
-M_{12} &-M_{21}
\end{array} \right), & \hbox{if $w_1>w_2$;} \\
          \frac{1}{\sqrt{2(M_{12}^2+M_{21}^2)}} \left(\begin{array}{cc}
-M_{12} & -M_{21} \\
M_{12} &M_{21}
\end{array} \right), & \hbox{if $w_2>w_1.$}
        \end{array}
      \right.
\end{eqnarray}
Using $\mathbf{u}_1=Qm^*\mathbf{h}$ and following calculations similar to those immediately succeeding (\ref{optcase}) in section \ref{example-analytic-2x2}, we obtain
\begin{eqnarray}
\mathbf{c}^\top\mathbf{u}_1 = \left\{
        \begin{array}{ll}
          \frac{\sqrt{M_{12}^2+M_{21}^2}}{\sqrt{2}(M_{12}+M_{21})^2}(c_1- c_2), & \hbox{if $w_1>w_2$;} \\
          \frac{\sqrt{M_{12}^2+M_{21}^2}}{\sqrt{2}(M_{12}+M_{21})^2}(c_2 - c_1), & \hbox{if $w_2>w_1$.}
        \end{array}
      \right.
\end{eqnarray}

\subsubsection{Analytic Solution for $M=\frac{1}{n}1_{n\times n}$}
Suppose that $M=\frac{1}{n}1_{n\times n}$ and that we would like to solve the problem (\ref{obj_lin_fun})-(\ref{zero_lin_fun}) for $\mathbf{c}\in\R^n$, $\mathbf{c}\neq a\textbf{1}$, where $a\in\R$. In this case, we have that $\textbf{h} = \frac{1}{n}\textbf{1}$ and $Q = I$; thus we have that $\textbf{w}=Q^\top \mathbf{c} = \mathbf{c}$. Also note that in this case, $N=\emptyset$ and $|N^c_j|=n$. Thus, we obtain
\begin{equation}
\begin{aligned}
m^*_{ij} &= \frac{1}{2\nu n}\left(c_i-\frac{1}{n}\sum_{k=1}^nc_k\right)\\
&=
\frac{c_i-\bar{\mathbf{c}}}{2\nu n},
\end{aligned}
\end{equation}
where $\bar{\mathbf{c}}=\frac{1}{n}\sum_{i=1}^n c_i$. Using the constraint $\|m^*\|_F^2 = 1$ we get that $2\nu = \theta\sigma(\mathbf{c})$, where $\theta\in\{-1,1\}$  and $\sigma(\mathbf{c})=\left(\frac{1}{n}\sum_{i=1}^n(c_i-\bar{\mathbf{c}})^2\right)^{1/2}$. We require $\nu >0$ for the solution to be a maximiser. As $\sigma(\mathbf{c})>0$, this requirement implies that $\theta = 1$. Thus, we finally have that
\begin{equation}
m^* = \frac{1}{n\sigma(\mathbf{c})}(\mathbf{c}\mathbf{1}^\top - \mathbf{\bar{c}}1_{n\times n})
\end{equation}
and
\begin{equation}
\mathbf{c}^\top \mathbf{u}_1 = \mathbf{c}^\top Qm^*\mathbf{h} = \frac{1}{n^2\sigma(\mathbf{c})}\left(n\|\mathbf{c}\|_2^2 - n\bar{\mathbf{c}}\sum_{i=1}^nc_i \right) = \frac{1}{n\sigma(\mathbf{c})}\left(\|\mathbf{c}\|_2^2 - n\left(\bar{\mathbf{c}}\right)^2\right).
\end{equation}

\section{Maximising the linear response of the rate of convergence to equilibrium}
\label{sect:rate}
In this section, we consider maximizing the linear response of the rate of convergence of the Markov chain to its equilibrium measure.
We achieve this by maximizing the linearised change in the magnitude of the second eigenvalue $\lambda_2$ of the stochastic matrix $M$.
The computations in this section will be applied in section \ref{sect:numeric} to solve a discrete version of the problem of maximizing the linear response of the rate of convergence to equilibrium for some stochastically perturbed dynamical system.
A related perturbative approach \cite{FrSa} increases the mixing rate of (possibly periodically driven) fluid flows by perturbing the advective part of the dynamics and solving a linear program to increase the spectral gap of the generator (infinitesimal operator) of the flow. In \cite{FGTW} kernel perturbations related to those used in section \ref{sect:numeric} were optimised to drive a nonequilibrium density toward equilibrium by solving a convex quadratic program with linear constraints.

Because $M$ is aperiodic and irreducible, $\lambda_1=1$ is the only eigenvalue on the unit circle.
Let $\lambda_2\in \mathbb{C}$ be the eigenvalue of $M$ strictly inside the unit circle with largest magnitude.
Denote by $\mathbf{l}_2\in \mathbb{C}^n$ and $\mathbf{r}_2\in\mathbb{C}^n$ the left and right eigenvectors of $M$ corresponding to $\lambda_2$.
We assume that we have the normalisations $\mathbf{r}_2^*\mathbf{r}_2=1$ and $\mathbf{l}_2^*\mathbf{r}_2=1$.
Considering the small perturbation of $M$ to $M+\varepsilon m$, by standard arguments (e.g.\ Theorem 6.3.12 \cite{hornjohnson}), one has
\begin{equation}
\label{dlambda}
\frac{d\lambda_2(\varepsilon)}{d\varepsilon}\bigg|_{\varepsilon = 0}=\mathbf{l}_2^*m\mathbf{r}_2,
\end{equation}
where $\lambda_2(\varepsilon)$ is the second largest eigenvalue of $M+\varepsilon m$.
We wish to achieve a maximal decrease in the magnitude of $\lambda_2$, or equivalently a maximal decrease in the real part of the logarithm of $\lambda_2$.
Denote by $\Re(\cdot)$ and $\Im(\cdot)$ the real and imaginary parts, respectively.
Now $d(\Re(\log\lambda_2(\varepsilon)))/d\varepsilon=\Re(d\log(\lambda_2(\varepsilon))/d\varepsilon)=\Re((d\lambda_2(\varepsilon)/d\varepsilon)/\lambda_2(\varepsilon))$, which, using (\ref{dlambda}) becomes
\begin{eqnarray}
\label{dlambda2}
\lefteqn{\Re((d\lambda_2(\varepsilon)/d\varepsilon)/\lambda_2)|_{\varepsilon =0}}\\
\nonumber&=&
\frac{\left(\Re(\mathbf{l}_2)^\top m \Re(\mathbf{r}_2)+\Im(\mathbf{l}_2)^\top m \Im(\mathbf{r}_2)\right)\Re(\lambda_2)+\left(\Re(\mathbf{l}_2)^\top m \Im(\mathbf{r}_2)-\Im(\mathbf{l}_2)^\top m \Re(\mathbf{r}_2)\right)\Im(\lambda_2)}{|\lambda_2|^2}.
\end{eqnarray}

Similarly to Section \ref{sect:expectation} we now have the optimisation problem:
\begin{eqnarray}
\label{obj_eval2}
\hspace*{-15cm}\min_{m\in\R^{n\times n}} && \hspace*{-.7cm}\left(\Re(\mathbf{l}_2)^\top m \Re(\mathbf{r}_2)+\Im(\mathbf{l}_2)^\top m \Im(\mathbf{r}_2)\right)\Re(\lambda_2)+\left(\Re(\mathbf{l}_2)^\top m \Im(\mathbf{r}_2)-\Im(\mathbf{l}_2)^\top m \Re(\mathbf{r}_2)\right)\Im(\lambda_2)\\
\label{stoch_lin_fun_eval2}\hspace*{-15cm}\mbox{subject to} && m^\top\mathbf{1} =\mathbf{0}\\
\label{norm_lin_fun_eval2}&&\|m\|_F^2-1=0\\
\label{zero_lin_fun_eval2}&& m_{ij} = 0 \text{ if } (i,j)\in N,
\end{eqnarray}
where $N = \{(i,j)\in \{1,\ldots,n\}^2: M_{ij}= 0 \text{ or } 1\}$. Note that as  $m_{ij}$ takes the value 0 for all $(i,j)\in N$, we just need to solve (\ref{obj_eval2})--(\ref{norm_lin_fun_eval2}) for $(i,j)\not\in N$.

Applying Lagrange multipliers, we proceed as in Section \ref{sect:expectation}, with the only change being to replace the expression (\ref{eq72}) with
\begin{equation}
\label{eq72a}
S_{ij}-\varrho_j-2\nu m_{ij} = 0\text{ for }(i,j)\not\in N,
\end{equation}
where \begin{equation}
\label{70a}
S_{ij}:=\left(\Re(\mathbf{l}_2)_i\Re(\mathbf{r}_2)_j+\Im(\mathbf{l}_2)_i\Im(\mathbf{r}_2)_j\right)\Re(\lambda_2)+\left(\Re(\mathbf{l}_2)_i\Im(\mathbf{r}_2)_j-\Im(\mathbf{l}_2)_i\Re(\mathbf{r}_2)_j\right)\Im(\lambda_2).
\end{equation}
Following the steps in section \ref{sect:expectation} we obtain
\begin{equation}
\label{70b}
m_{ij}^* = \frac{-\varrho_j+S_{ij}}{2\nu} = \frac{\left(S_{ij}-\frac{1}{|N_j^c|}\sum_{l:(l,j)\not\in N}S_{lj}\right)}{2\nu},
\end{equation}
where $(i,j)\not\in N$ and $N_j^c = \{i: (i,j)\not\in N\}$. Next, we use a similar argument as in section \ref{sect:expectation} to select the correct sign of $\nu$. As our objective function is linear, the only non-linear term in the Lagrangian for this problem is from constraint ~\eqref{zero_lin_fun_eval2}; thus, as in section \ref{sect:expectation}, we again have that
\begin{equation*}
\frac{\partial^2 L}{\partial m_{ij}\partial m_{kl}}(m^*,\boldsymbol\varrho,\nu) = -2\nu\delta_{(i,j),(k,l)}.
\end{equation*}
Thus, the Hessian matrix of the Langrangian function is $H(m^*,\boldsymbol\varrho,\nu) = -2\nu$ Id$_{n^2-|N|}$. Using Theorem 9.3.2 \cite{fletcher1991practical} and the argument for the necessary condition for maximisation in section \ref{sect:expectation}, we conclude that imposing $\nu<0$ will ensure that the solution is the minimiser.

Using the constraint (\ref{norm_lin_fun_eval2}), we select $\nu$ to ensure $\|m^*\|_F=1$. Writing $m^*_{ij}=\frac{\widetilde{m}_{ij}}{\nu}$, the constraint $\|m^*\|_F=1$ will give us $\nu =\theta\|\widetilde{m}\|_F$, where $\theta\in\{-1,1\}$. As we require that $\nu <0$ for the solution to be the minimiser, we conclude that $\theta=-1$ and
\begin{equation}\label{nu-Mixing}
\nu = -\|\widetilde{m}\|_F
\end{equation}
\subsection{Algorithm}\label{Algo-Mix}
The following algorithm can be used to compute the optimal perturbation $m^*$ to maximise the linear response of the rate of convergence to equilibrium.
\begin{multicols}{2}
\textbf{Algorithm 4}
\begin{enumerate}
\item Compute $\textbf{h}$ as the invariant probability vector of $M$. Compute $\textbf{r}_2$ and $\textbf{l}_2$, the right and left eigenvectors corresponding to the second largest eigenvalue of $M$,  normalised as $\textbf{r}_2^*\mathbf{r}_2=1$ and $\textbf{l}_2^*\mathbf{r}_2=1$.
\item Construct the matrix $S$ from (\ref{70a}).
\item Calculate $m^*_{ij}$ according to (\ref{70b}), where $\nu$ is given by ~\eqref{nu-Mixing}.

\end{enumerate}
\columnbreak
\textbf{Matlab Code}
\begin{verbatim}
function m = lin_resp_eval2(M)
%Step 1
[V,D] = eigs(M,2);
if abs(D(2,2))>abs(D(1,1))
    V(:,[1,2]) = V(:,[2,1]);
    D(:,[1,2]) = D(:,[2,1]);
end
h = V(:,1);
h = h/sum(h);
r = V(:,2);
[V1,D1] = eigs(M',2);
if abs(D1(2,2))>abs(D1(1,1))
    V1(:,[1,2]) = V1(:,[2,1]);
    D1(:,[1,2]) = D1(:,[2,1]);
end
l = V1(:,2);
l = (1/(conj(l)'*r))*V1(:,2);
%Step 2
d = D(2,2);
S=real(d)*(real(l)*real(r)'+imag(l)*imag(r)')...
	+imag(d)*(real(l)*imag(r)'-imag(l)*real(r)');
%Step 3
n=length(M);
m = zeros(n);
for i=1:n
    K = find(M(:,i)>10^-7);
    if(length(K) > 1)
        m(K,i) = (S(K,i)- mean(S(K,i)));
    end
end
m = -m./(norm(m,'fro'));
end
\end{verbatim}
\end{multicols}

\subsection{Analytic example}

\subsubsection{Analytic Solution for $M\in\R^{2\times 2}$}
Suppose that $M\in\R^{2\times 2}$ and we would like to solve (\ref{obj_eval2})-(\ref{zero_lin_fun_eval2}). As in section \ref{example-analytic-2x2} for $M\in\R^{2\times 2}$, we only need to consider the case when $M$ is positive. Writing
\begin{equation*}
M = \left(\begin{array}{cc}
M_{11} & 1-M_{22} \\
1-M_{11} & M_{22}
\end{array}\right),
\end{equation*}
where $M_{11},M_{22}\in (0,1)$, we compute $\lambda_2 = M_{11}+M_{22}-1,
\mathbf{r}_2 =\frac{1}{\sqrt{2}} \left(\begin{array}{c}
-1\\
1
\end{array}\right)$ and $\mathbf{l}_2 = \frac{\sqrt{2}}{M_{11}+M_{22}-2}\left(\begin{array}{c}
1-M_{11} \\
M_{22}-1
\end{array}\right).$ Using these computations and ~\eqref{70a}, we have that
\begin{equation}
S_{11} = \lambda_2 \frac{M_{11}-1}{M_{11}+M_{22}-2} = -S_{12}\text{ and } S_{22} = \lambda_2 \frac{M_{22}-1}{M_{11}+M_{22}-2}=-S_{21}.
\end{equation}
With this and ~\eqref{70b}, we have that
\begin{equation}
m^*_{11} = -m^*_{21} = \frac{\lambda_2}{4\nu}\text{ and }m^*_{22} = -m^*_{12} = \frac{\lambda_2}{4\nu}.
\end{equation}
Using the constraint $\|m^*\|_F^2 = 1$, we get $2\nu = \theta \lambda_2$, where $\theta\in\{-1,1\}$. As we require $\nu<0$ for the solution to be the minimiser, if $\lambda_2<0$, i.e. $M_{11}+M_{22}<1$, then $\theta = 1$ and if $\lambda_2>0$, i.e. $M_{11}+M_{22}>1$, then $\theta = -1$. Thus, we have that
\begin{eqnarray}
m^* = \left\{
        \begin{array}{ll}
          \frac{1}{2}\left(\begin{array}{cc}
1 & -1 \\
-1 & 1
\end{array} \right), & \hbox{if $M_{11}+M_{22}<1$;} \\
          \frac{1}{2}\left(\begin{array}{cc}
-1 & 1 \\
1 & -1
\end{array} \right), & \hbox{if $M_{11}+M_{22}>1$.}
        \end{array}
      \right.
\end{eqnarray}
Using ~\eqref{dlambda2} and the fact that $m^*, \mathbf{l}_2, \mathbf{r}_2$ and $\lambda_2$ are real, we finally obtain
\begin{eqnarray}
\frac{d(\Re(\log\lambda_2(\varepsilon)))}{d\varepsilon}\bigg|_{\varepsilon=0} = \frac{1}{\lambda_2}\mathbf{l}^*_2m^*\mathbf{r}_2 = \left\{
        \begin{array}{ll}
          \frac{1}{M_{11}+M_{22}-1} = \frac{1}{\lambda_2}, & \hbox{if $M_{11}+M_{22}<1$;} \\
          \frac{-1}{M_{11}+M_{22}-1}= \frac{-1}{\lambda_2}, & \hbox{if $M_{11}+M_{22}>1$.}
        \end{array}
      \right.
\end{eqnarray}

\section{Optimizing linear response for a general sequence of matrices}
\label{sect:sequential}
In this section we extend the ideas of Sections \ref{sect:l2} and \ref{sect:expectation} to derive the linear response of the Euclidean norm of the invariant probability vector $\mathbf{h}$ and the expectation of an observable $\mathbf{c}$ , when acted on by a \emph{finite sequence of matrices}.
We will then introduce and solve an optimization problem which finds the sequence of perturbation matrices that achieve these maximal values.
\subsection{Linear response for the invariant measure}
Let $M^{(0)},M^{(1)},\dots,M^{(\tau-1)}$ be  a fixed finite sequence of column stochastic matrices.
Furthermore, let $m^{(t)}$,  $t\in \{0,\ldots,\tau-1\}$ be a  sequence of perturbation matrices. Take an arbitrary probability vector $\textbf{h}^{(0)}$ and set
\[
\textbf{h}^{(t+1)}=M^{(t)}\textbf{h}^{(t)}, \quad \text{for $t\in \{0,\ldots,\tau-1\}.$}
\]
We now want to derive the formula for the linear response of $\textbf{h}^{(\tau)}$. We require that
\begin{equation}\label{seq1} (M^{(t)}+\varepsilon m^{(t)})\bigg{(}\textbf{h}^{(t)}+\sum_{i=1}^{\infty}\varepsilon^i \textbf{u}^{(t)}_i \bigg{)} = \textbf{h}^{(t+1)}+\sum_{i=1}^{\infty}\varepsilon^i \textbf{u}^{(t+1)}_i, \end{equation}
where $\varepsilon\in\R$. We refer to $\textbf{u}^{(t)}_1$ as the linear response at time $t$.  By expanding the left-hand side of~\eqref{seq1}, we have
\begin{equation}\label{seq2}\left(M^{(t)}+\varepsilon m^{(t)}\right)\bigg{(}\textbf{h}^{(t)}+\sum_{i=1}^{\infty}\varepsilon^i \textbf{u}^{(t)}_i\bigg{)}=\textbf{h}^{(t+1)}+\varepsilon \left(M^{(t)}\textbf{u}^{(t)}_1+m^{(t)}\textbf{h}^{(t)}\right) + O(\varepsilon^2).\end{equation}
 Denoting for simplicity $\textbf{u}^{(t)}_1$ by $\textbf{u}^{(t)}$, it follows from~\eqref{seq1} and~\eqref{seq2}  that
 \begin{equation}\label{iter}
\textbf{u}^{(t+1)} = M^{(t)}\textbf{u}^{(t)}+m^{(t)}\textbf{h}^{(t)}.
\end{equation}
Set $\textbf{u}^{(0)}=0$. Iterating~\eqref{iter}, we obtain  that
\begin{equation}\label{lin-resp-seq}
\textbf{u}^{(\tau)} = \sum_{t=1}^{\tau-1}M^{(\tau-1)}\dots M^{(t)}m^{(t-1)}\textbf{h}^{(t-1)} + m^{(\tau-1)}\textbf{h}^{(\tau-1)}.
\end{equation}

\subsubsection{The optimization problem}
It follows from Proposition \ref{Prop-Kron-Prod}(vi) that
\begin{equation*}
\begin{aligned}
\textbf{u}^{(\tau)} = \widehat{\textbf{u}}^{(\tau)} &= \sum_{t=1}^{\tau-1} \left(\textbf{h}^{(t-1)\top}\otimes \left(M^{(\tau-1)}\cdots M^{(t)}\right)\right)\widehat{m}^{(t-1)} +(\textbf{h}^{(\tau-1)\top}\otimes \text{Id})\widehat{m}^{(\tau-1)}\\
&= \sum_{t=1}^{\tau-1}W^{(t-1)}\widehat{m}^{(t-1)}+W^{(\tau-1)} \widehat{m}^{(\tau-1)} \\
&= W\left(\begin{array}{c}
\widehat{m}^{(0)} \\
\vdots \\
\widehat{m}^{(\tau-1)}
\end{array} \right) = W\widehat{m},
\end{aligned}
\end{equation*}
where
\[
 W^{(t)}=\textbf{h}^{(t)\top}\otimes \left(M^{(\tau-1)}\cdots M^{(t+1)}\right) \quad \text{for $0\le t \le \tau-2$,} \quad W^{(\tau-1)}=\textbf{h}^{(\tau-1)\top}\otimes \text{Id}
 \]
and
\[
 W=\left(W^{(0)}|W^{(1)}|\dots|W^{(\tau-1)}\right).
\]
Note that the $W^{(t)}$s are $n\times n^2$ matrices, $W$ is an $n\times \tau n^2$ matrix and $\widehat{m}$ is a $\tau n^2$-vector.

We consider the following optimization problem, which maximises the response of the Euclidean norm of the response $\mathbf{u}^{(\tau)}$:
\begin{eqnarray}
\label{obj6}
\max_{\widehat{m}\in\R^{\tau n^2}} &&\|W\widehat{m}\|_2^2\\
\label{stoch3}\mbox{subject to} && A^{(t)}\widehat{m}^{(t)} = \textbf{0}\text{ for } t=0,\dots,\tau-1\\
\label{norm6}&&\sum_{t=0}^{\tau-1}\|\widehat{m}^{(t)}\|_2^2-1=0,
\end{eqnarray}
where $A^{(t)}$ is the constraint matrix (\ref{A-final}) associated to the matrix $M^{(t)}$ and conditions~\eqref{stoch2} and~\eqref{zeros}.

\subsubsection{Solution to the optimization problem}
We  want to reformulate the optimization problem with the constraints (\ref{stoch3}) removed.
We first note that~\eqref{stoch3} can be replaced by $A\widehat{m} =\textbf{0}$, where
\begin{equation}\label{A-final-seq}
A = \text{diag}(A^{(0)},\dots,A^{(\tau-1)}).
\end{equation}
Let $E^{(t)}$ be an $n^2\times \ell^{(t)}$ matrix whose columns form an  orthonormal basis of the null space of $A^{(t)}$ for $t=0,\dots,\tau-1$,  where
$\ell^{(t)}$ denotes the nullity of $A^{(t)}$. Then,
\begin{equation*}
E = \text{diag}(E^{(0)},\dots,E^{(\tau-1)})
\end{equation*}
is a matrix whose columns form an  orthonormal basis  of the null space of the matrix $A$ in (\ref{A-final-seq}).
Thus, if $\widehat{m}$ is an element of the null space of $A$ then,  $\widehat{m} = E\boldsymbol\alpha$ for a unique $\boldsymbol\alpha\in\R^{\sum_{t=0}^{\tau-1}\ell^{(t)}}$.
Finally, as $$\sum_{t=0}^{\tau-1}\|\widehat{m}^{(t)}\|_2^2 = \|\widehat{m}\|_2^2 = \|E\boldsymbol\alpha\|_2^2 = \|\boldsymbol\alpha\|_2^2,$$
we can reformulate the optimization problem (\ref{obj6})-(\ref{norm6}) as:
\begin{eqnarray}
\label{obj7}
\max_{\boldsymbol\alpha\in\R^{\sum_{t=0}^{\tau-1}\ell^{(t)}}} &&\|U\boldsymbol\alpha\|_2^2\\
&&\|\boldsymbol\alpha\|_2^2-1=0,
\end{eqnarray}
where \begin{equation}\label{seqU} U = WE = (W^{(0)}E^{(0)}|\dots|W^{(k-1)}E^{(k-1)}).\end{equation}
Arguing as in section \ref{sect-gen-final-soln}, we conclude that $\widehat{m}^* = E\boldsymbol\alpha^*$ maximises the Euclidean norm of the linear response $\textbf{u}^{(\tau)}$,
where $\boldsymbol\alpha^*\in\R^{\sum_{t=0}^{\tau-1}\ell^{(t)}}$ is the eigenvector corresponding to the largest eigenvalue of $U^\top U$ (with $U$ as in~\eqref{seqU}). Finally, if we denote $\mathbf{h}^{(t+1)}(\varepsilon) = \left(M^{(t)}+\varepsilon m^{(t),*}\right)\mathbf{h}^{(t)} $, we choose the sign of $m^{(t),*}$ so that $\|\mathbf{h}^{(t)}\|_2<\|\mathbf{h}^{(t)}(\varepsilon)\|_2$ for small $\varepsilon>0$ and for each $t\in\{1,\dots,\tau\}$; this is possible as $\textbf{h}^{(t)}$ is independent of $m^{(t)}$.

\subsection{Linear response for the expectation of an observable}
 In this section, we consider maximising the linear response of the expected value of an observable $\mathbf{c}$ with respect to the probability vector $\mathbf{h}^{(\tau)}$, when acted on by a \emph{finite sequence of matrices}.
 More explicitly, we consider the following problem: For $\mathbf{c}\in\R^n$
\begin{eqnarray}
\label{obj7a}
\max_{m^{(0)},m^{(1)},\ldots,m^{(\tau-1)}\in\R^{n\times n}} &&\mathbf{c}^\top\textbf{u}^{(\tau)}\\
\label{stoch7}\mbox{subject to} && m^{(t)\top}\textbf{1} = \textbf{0}\text{ for } t\in\{0,\dots,\tau-1\}\\
\label{norm7}&&\sum_{t=0}^{\tau-1}\|{m}^{(t)}\|_F^2-1=0\\
\label{zeros7}&&m^{(t)}_{ij}=0\text{ if }(i,j)\in N^{(t)}\text{ for } t\in\{0,\dots,\tau-1\},
\end{eqnarray}
where $\textbf{u}^{(t)}$ is the linear response at time $t$, $m^{(t)}_{ij}$ is the $(i,j)$ element of the matrix $m^{(t)}$ and $N^{(t)} = \{(i,j)\in \{1,\ldots,n\}^2: M^{(t)}_{ij} = 0\text{ or }1\}$. Multiplying ~\eqref{lin-resp-seq} on the left by $\mathbf{c}^\top$ we obtain
\begin{equation*}
\mathbf{c}^\top\mathbf{u}^{(\tau)} = \sum_{t=0}^{\tau-1}\mathbf{w}^{(t)\top}m^{(t)}\textbf{h}^{(t)},
\end{equation*}
where $\textbf{w}^{(t)\top}=\mathbf{c}^\top M^{(\tau-1)}\dots M^{(t+1)}$ for $t\in \{0,\dots,\tau-2\}$ and $\textbf{w}^{(\tau-1)\top} = \mathbf{c}^\top$. Note that as the values of $m^{(t)}_{ij}=0$ for $(i,j)\in N^{(t)}$, we just need to solve (\ref{obj7a})--(\ref{norm7}) for $(i,j)\not\in N^{(t)}$.

As in section \ref{sect:expectation}, we solve this problem using the method of Lagrange multipliers. We begin by considering the following Lagrangian function:
\begin{equation}\label{Lag-fun-lin-fun-seq}
L(m^{(0)},\dots,m^{(\tau-1)},\boldsymbol\varrho^{(0)},\dots,\boldsymbol\varrho^{(\tau-1)},\nu) = \sum_{t=0}^{\tau-1}\textbf{w}^{(t)\top}m^{(t)}\textbf{h}^{(t)} - \sum_{t=0}^{\tau-1}\boldsymbol\varrho^{(t)^\top} m^{(t)^\top}\textbf{1}-\nu\left(\sum_{t=0}^{\tau-1}\|m^{(t)}\|_F^2-1\right),
\end{equation}
where $\boldsymbol\varrho^{(t)}\in\R^n$ and $\nu\in\R$ are the Lagrange multipliers. Differentiating ~\eqref{Lag-fun-lin-fun-seq} with respect to $m^{(t)}_{ij}$, we get that
\begin{equation*}
\frac{\partial L}{\partial m^{(t)}_{ij}} (m^{(0)},\dots,m^{(\tau-1)},\boldsymbol\varrho^{(0)},\dots,\boldsymbol\varrho^{(\tau-1)},\nu) = w^{(t)}_{i}h^{(t)}_{j}-\varrho^{(t)}_{j}-2\nu m^{(t)}_{ij},
\end{equation*}
where $w^{(t)}_{i},h^{(t)}_{j},\varrho^{(t)}_{j}\in\R $ are the elements of the $n$-vectors $\textbf{w}^{(t)},\textbf{h}^{(t)}$ and $ \boldsymbol\varrho^{(t)}$ respectively.

Using the method of Lagrangian multipliers, we want that
\begin{equation*}
w^{(t)}_{i}h^{(t)}_{j}-\varrho^{(t)}_{j}-2\nu m^{(t)}_{ij} = 0\text{ for }(i,j)\not\in N^{(t)}
\end{equation*}
and
\begin{equation}\label{stoch_lin_fun_2_seq}
\sum_{i:(i,j)\not\in N^{(t)}}m^{(t)}_{ij} = 0\text{ for }j\in\{1,\dots,n\}, t\in\{0,\dots, \tau-1\}.
\end{equation}

We have that $\varrho^{(t)}_{j}= w^{(t)}_{i}h^{(t)}_{j}-2\nu m^{(t)}_{ij}$. Using ~\eqref{stoch_lin_fun_2_seq}, we obtain
\begin{equation*}
\sum_{i:(i,j)\not\in N^{(t)}}\varrho^{(t)}_{j}:=\left|N^{(t),c}_j\right|\varrho^{(t)}_{j} = h^{(t)}_{j}\sum_{i:(i,j)\not\in N^{(t)}}w^{(t)}_{i},
\end{equation*}
where $N_j^{(t),c} = \{i: (i,j)\not\in N^{(t)}\}$. Thus, we get that
\begin{equation*}
m^{(t),*}_{ij} = \frac{h^{(t)}_{j}\theta^{(t)}}{2\nu}\left(w^{(t)}_{i} - \frac{1}{\left|N^{(t),c}_j\right|}\sum_{i:(i,j)\not\in N^{(t)}}w^{(t)}_{i} \right),
\end{equation*}
where $(i,j)\not\in N^{(t)},\theta^{(t)}\in\{-1,1\}$ and $\nu>0$. Note that $\mathbf{w}^{(t)}$ and $\mathbf{h}^{(t)}$ are independent of $m^{(t)}$ and so we choose the sign of $m^{(t),*}$ by selecting $\theta^{(t)}$ such that $\mathbf{w}^{(t)\top}m^{(t),*}\mathbf{h}^{(t)}\ge 0$.  Finally, the normalisation $\nu>0$ is selected to ensure that we satisfy constraint ~\eqref{norm7}.

\section{Numerical Examples of optimal linear response for stochastically perturbed dynamical systems}
\label{sect:numeric}
We apply the techniques we have developed in Sections \ref{sect:l2}--\ref{sect:rate} to randomly perturbed dynamical systems.
We consider random dynamical systems of the form $x_{t+1}=T(x_t)+\xi_t$, $t=1,2,\ldots$, where $X\subset\mathbb{R}^d$, $T:X\to X$ is a measurable map on the phase space $X$ and the $\{\xi_k\}$ are i.i.d.\ random variables taking values in $X$ distributed according to a density $q:X\to\mathbb{R}^+$.
Suppose that $x_t$ is distributed according to the density $f_t:X\to \mathbb{R}^+$.
By standard arguments (e.g.\ \S10.5 \cite{lasota1985probabilistic}) one derives that $x_{t+1}$ is distributed according to the density $f_{t+1}(x)=\int_X q(x-T(y))f_t(y)\ dy$.
We thus define the annealed Perron-Frobenius (or transfer) operator
\begin{equation}
\label{annealedop}
\mathcal{P}f(x)=\int_X q(x-T(y))f(y)\ dy
 \end{equation}
 as the linear (Markov) operator that pushes forward densities under the annealed action of our random dynamical system.
More generally, writing $q(x-T(y)=k(x,y)$, we think of $k$ as a kernel defining the integral operator $\mathcal{P}f(x)=\int_X k(x,y)f(y)\ dy$.
We will assume that $k\in L^2(X\times X)$, which guarantees that $\mathcal{P}$ is a compact operator on $L^2(X)$;  see e.g.\ Proposition II.1.6 \cite{conway}.
A sufficient condition for $\mathcal{P}$ possessing a unique fixed point in $L^1$ is that there exists a $j$ such that $\int_X \inf_y k^{(j)}(x,y)\ dx>0$, where $k^{(j)}$ is the kernel associated with $\mathcal{P}^j$;  see Corollary 5.7.1 \cite{lasota1985probabilistic}.
This is a stochastic ``covering'' condition, which is satisfied by our examples, which are generated by transitive deterministic $T$ with bounded additive uniform noise.
In summary, we have a unique annealed invariant measure for our stochastically perturbed system and by compactness our transfer operator $\mathcal{P}$ has a spectral gap on $L^2(X)$.


\subsection{Ulam projection}
In order to carry out numerical computations, we project the operator $\mathcal{P}$ onto a finite-dimensional space spanned by indicator functions on a fine mesh of $X$.
Let $B_n=\{I_1,\ldots,I_n\}$ denote a partition of $X$ into connected sets, and  set $\mathcal{B}_n=\mbox{span}\{\mathbf{1}_{I_1},\ldots,\mathbf{1}_{I_n}\}$.
Define a projection $\pi_n:L^1(X)\to \mathcal{B}_n$ by $\pi_n(f)=\sum_{i=1}^n \left( \frac{1}{\ell(I_i)}\int_{I_i} f\ dx \right) \mathbf{1}_{I_i}$, where $\ell$ is Lebesgue measure; $\pi_n$ simply replaces $f|_{I_i}$ with its expected value.
We now consider the finite-rank  operator $\pi_n\mathcal{P}:L^1\to\mathcal{B}_n$;  this general approach is known as Ulam's method \cite{ulambook}.

We calculate
\begin{eqnarray}
\nonumber\pi_n\mathcal{P}f&=&\sum_{i=1}^n \left( \frac{1}{\ell(I_i)}\int_{I_i} \mathcal{P}f\ dx \right) \mathbf{1}_{I_i}\\
\nonumber&=&\sum_{i=1}^n \left( \frac{1}{\ell(I_i)}\int_{I_i} \int_X q(x-T(y))f(y)\ dy\ dx \right) \mathbf{1}_{I_i}\\
\nonumber&=&\sum_{i=1}^n \left( \frac{1}{\ell(I_i)} \int_X \underbrace{\int_{I_i} q(x-T(y))\ dx}_{:=\psi_i(y)} f(y)\ dy \right) \mathbf{1}_{I_i}\\
\label{piP1}&=&\sum_{i=1}^n \frac{\int_X \psi_i(y) f(y)\ dy}{\ell(I_i)}\mathbf{1}_{I_i}.
\end{eqnarray}
Putting $f=\sum_{j=1}^n f_j\mathbf{1}_{I_j}\in \mathcal{B}_n$, where $f_j\in\mathbb{R}, j=1,\ldots,n$, we have
\begin{eqnarray}
\nonumber\pi_n\mathcal{P}f&=&\sum_{i=1}^n\sum_{j=1}^n f_j \frac{\int_X \psi_i(y) \mathbf{1}_{I_j}(y)\ dy}{\ell(I_i)}\mathbf{1}_{I_i}\\
\label{piP2}&=&\sum_{i=1}^n\sum_{j=1}^n f_j \underbrace{\frac{\int_{I_j} \psi_i(y)\ dy}{\ell(I_i)}}_{:=M_{ij}}\mathbf{1}_{I_i},
\end{eqnarray}
where $M$ is the matrix representation of $\pi_n\mathcal{P}:\mathcal{B}_n\to\mathcal{B}_n$.

In our examples below, $X=[0,1]$ or $X=S^1$, and $q(x-Ty)=k(x,y):=\mathbf{1}_{B_\epsilon(Ty)}(x)/\ell(X\cap B_\epsilon(Ty))$, where $B_\epsilon(Ty)$ denotes an $\epsilon$-ball centred at the point $Ty$.
We require this slightly more sophisticated version of $q$ in order to ensure that we do not stochastically perturb points outside our domain $X$.
Our random dynamical systems therefore comprise deterministic dynamics followed by the addition of uniformly distributed noise in an $\epsilon$-ball (with adjustments made near the boundary of $X$).
This choice of $q$ leads to
\begin{equation}
\label{piP3}\psi_i(y)=\frac{\int_{I_i} \mathbf{1}_{B_\epsilon(Ty)}(x)\ dx}{\ell(X\cap B_\epsilon(Ty))}=\frac{\ell(I_i\cap B_\epsilon(Ty))}{\ell(X\cap B_\epsilon(Ty))}.
\end{equation}
Combining (\ref{piP2}) and (\ref{piP3}) we obtain
$$M_{ij}=\frac{\int_{I_j} \ell(I_i\cap B_\epsilon(Ty))/\ell(X\cap B_\epsilon(Ty))\ dy}{\ell(I_i)}.$$
From now on we assume that $I_i=[(i-1)/n,i/n), i=1,\ldots,n$, so that $\mathcal{B}_n$ is an partition of $X$ into equal length subintervals.
We now have that $\sum_{i=1}^n M_{ij}=1$ for each $j=1,\ldots,n$, and so $M$ is a column stochastic matrix.
We use the matrix $M$ to numerically approximate the operator $\mathcal{P}$ in the experiments below.

\subsubsection{Consistent scaling of the perturbation $m$}
In sections \ref{sect:lanford}--\ref{sect:doublelanford} we will think of the entries of the perturbation matrix $m$ as resulting from the matrix representation of the Ulam projection of a perturbation $\delta\mathcal{P}$ of $\mathcal{P}$.
To make this precise, we first write $f\in\mathcal{B}_n$ as
$f=\sum_{j=1}^n \bar{f}_j\mathbf{1}_{I_j}$, and introduce a projected version of $\delta k$: $\pi_n(\delta k)=\sum_{i,j}\bar{\delta k}_{ij}\mathbf{1}_{I_i\times I_j}$, where the matrix $\bar{\delta k}_{ij}=(1/(\ell(I_i)\ell(I_j)))\int_{I_i\times I_j}\delta k(x,y)\ dy dx$.
We now explicitly compute the Ulam projection of $\delta P$:

\begin{eqnarray*}
\pi_n\delta \mathcal{P}(f)(z)&=&(1/\ell(I_i))\sum_{i=1}^n \left[ \int_{I_i}\delta \mathcal{P}(f)(x)\ dx\right] \mathbf{1}_{I_i}(z)\\
&=&(1/\ell(I_i))\sum_{i=1}^n \left[ \int_{I_i}\int_X\delta k(x,y)f(y)\ dy dx\right] \mathbf{1}_{I_i}(z)\\
&=&(1/\ell(I_i))\sum_{i=1}^n \left[ \int_{I_i}\sum_{j=1}^n\int_{I_j}\delta k(x,y)\bar{f}_j\ dy dx\right] \mathbf{1}_{I_i}(z)\\
&=&(1/\ell(I_i))\sum_{i,j=1}^n \bar{f}_j \left[ \int_{I_i\times I_j}\delta k(x,y)\ dy dx\right] \mathbf{1}_{I_i}(z)\\
&=&\sum_{i,j=1}^n \underbrace{\ell(I_j)\bar{\delta k}_{ij}}_{:=m_{ij}}\bar{f}_j \mathbf{1}_{I_i}(z)
\end{eqnarray*}
Thus, we have the relationship $m_{ij}=\ell(I_j)\bar{\delta k}_{ij}$ between the matrix representation of the projected version of the operator $\delta \mathcal{P}$ (namely $m$) and the elements of the projected version of the kernel (namely $\bar{\delta k}$).

We wish to fix the Hilbert-Schmidt norm of $\pi_n\delta \mathcal{P}$ to 1.
\begin{eqnarray}
\nonumber 1=\|\pi_n\delta \mathcal{P}\|_{HS}^2&=&\|\pi_n\bar{\delta k}\|_{L^2(X\times X)}^2\\
\nonumber&=&\left\|\sum_{i,j=1}^n\bar{\delta k}_{ij}^21_{I_i\times I_j}\right\|_{L^2(X\times X)}^2 \\
\nonumber&=&\sum_{i,j=1}^n\int_{I_i\times I_j}\bar{\delta k}^2dydx\\
\label{HSeq}&=&\sum_{i,j=1}^n\ell(I_i)\ell(I_j)\bar{\delta k}^2.
\end{eqnarray}
Since $\|m\|_F^2=\sum_{i,j=1}^n \ell(I_j)^2\bar{\delta k}_{ij}^2$, if we assume that $\ell(I_i)=1/n$, $1\le i\le n$, we obtain $\|m\|_F=(1/n)^2\|\bar{\delta k}\|_F^2$ and by (\ref{HSeq}) we know $\|\bar{\delta k}\|_F^2=n^2$.
We thus conclude that enforcing $\|m\|_F=1$ will ensure $\|\pi_n\delta\mathcal{P}\|_{HS}=1$, as required.

\subsubsection{Consistent scaling for $h$ and $f$}
In sections \ref{sect:lanford}--\ref{sect:doublelanford} we will use vector representations of the invariant density $h$ and an $L^2$ function $c$.
We write $h=\sum_{i=1}^n \mathbf{h}_i\mathbf{1}_{I_i}$, where $\mathbf{h}\in\mathbb{R}^n$.
We normalise so that $\int_X h(x)\ dx=1$, which means that $\sum_{i=1}^n \mathbf{h}_i=n$.
Similarly, we write $c=\sum_{i=1}^n \mathbf{c}_i\mathbf{1}_{I_i}$, where $\mathbf{c}\in\mathbb{R}^n$.
We normalise so that $\int_X c(x)^2\ dx=1$, which means that $\sum_{i=1}^n \mathbf{c}_i^2=n$ or $\|\mathbf{c}\|_{\ell^2}=\sqrt{n}$.

\subsection{A stochastically perturbed Lanford Map}
 \label{sect:lanford}
 The first example we consider is the stochastically perturbed Lanford map. We will use the numerical solution of the problems (\ref{obj5})-(\ref{zeros}) and (\ref{obj_lin_fun})-(\ref{zero_lin_fun}) for this map to solve the problem of maximising the $L^2$-norm of the linear response of the invariant measure and maximising the linear response of the expectation of an observable.
\subsubsection{Maximising the linear response of the $L^2$-norm of the invariant measure}
Let $T:S^1\rightarrow S^1$ be the stochastically perturbed Lanford map defined by
\begin{equation}\label{Lanford_Map_Noise}
T(x) = 2x +\frac{1}{2}x(1-x)+ \xi \text{ mod } 1,
\end{equation}
where $\xi\sim\mathcal{U}(0,\frac{1}{10})$ (uniformly distributed on an interval about 0 of radius 1/10). Let $M\in\R^{n\times n}$ be Ulam's discretization of the transfer operator of the map $T$ with $n$ partitions. Using Algorithm 2, we solve the problem (\ref{obj5})-(\ref{zeros}) for the matrix $M$ for $n=2000$ to obtain the optimal perturbation $m^*$. The top two singular values of the matrix $\widetilde{U}$, computed using MATLAB, are 0.0175 and 0.0167 (each with multiplicity one), which we consider to be strong numerical evidence that the leading singular value of $\widetilde{U}$ has multiplicity one. By Proposition \ref{Prop-Unique} we conclude that our computed $m^*$ is the unique optimal perturbation for the discretized system. The sign of the matrix $m^*$ is chosen so that $\|\textbf{h}_M\|_2<\|\textbf{h}_{M+\varepsilon m^*}\|_2$ for $\varepsilon>0$. Figure \ref{fig-Lanford-noise}(A) shows the Lanford map and figure \ref{fig-Lanford-noise}(B) presents the approximation of the invariant density $h$ of the Lanford map.
\begin{figure}[h]
\subfloat[Subfigure 1 list of figures text][Colourmap of the stochastically perturbed Lanford map. The colourbar indicates the values of the elements of the matrix.]{
\includegraphics[width=0.5\textwidth]{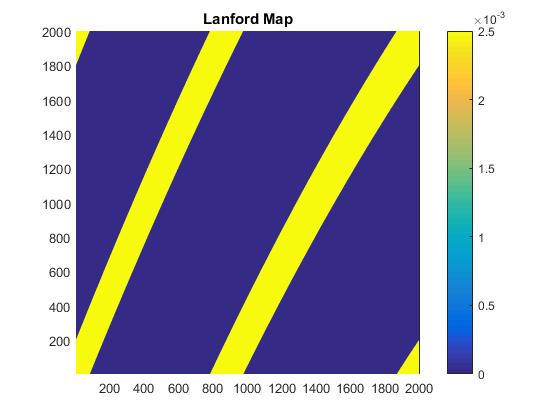}
\label{fig:subfig1}}
\subfloat[Subfigure 2 list of figures text][The invariant density $h$.]{
\includegraphics[width=0.5\textwidth]{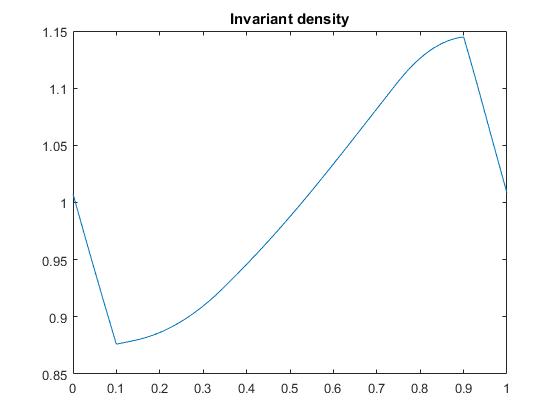}
\label{fig:subfig2}}
\qquad
\subfloat[Subfigure 3 list of figures text][The optimal perturbation $m^*$ for the problem. The colourbar indicates the values of the elements of the matrix.]{
\includegraphics[width=0.5\textwidth]{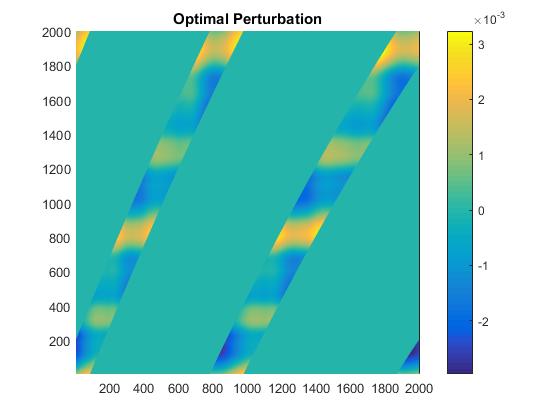}
\label{fig:subfig3}}
\subfloat[Subfigure 4 list of figures text][The optimal linear response $u^*_1$ of the invariant density.]{
\includegraphics[width=0.5\textwidth]{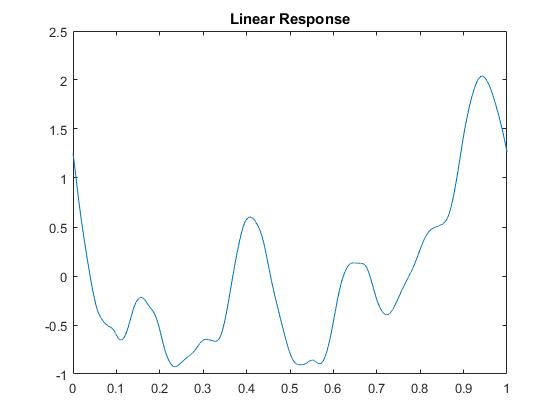}
\label{fig:subfig4}}
\caption{Solution to the problem of maximising the $L^2$-norm of the linear response of the stochastically perturbed Lanford map.}
\label{fig-Lanford-noise}
\end{figure}
Figure \ref{fig-Lanford-noise}(C) presents the optimal perturbation matrix $m^*$ which generates the maximal response. Figure \ref{fig-Lanford-noise}(D) presents the approximation of the associated linear response $u^*_1 = \sum_{i=1}^n\textbf{u}_1^*\textbf{1}_{I_i}$, for the perturbation $m^*$; for this example, we compute $\|u_1^*\|_{L^2}^2 \approx 0.6154$.
Figure \ref{fig-Lanford-noise}(C) shows that the selected perturbation preferentialy places mass in a neighbourhood of $x=0.4$ and $x=0.95$, consistent with local peaks in the response in Figure \ref{fig-Lanford-noise}(D).

Having computed the optimal linear response for a specific $n$, we verify in Table \ref{table-Lanford} that for various partition cardinalities, the $L^2$-norm of the approximation of the linear response $u_1^*$ converges. We also verify that $\|h_{M+\varepsilon m^*}-(h_M+\varepsilon u_1^*)\|_{L^2}^2$ is small for small $\varepsilon>0$. The 1000-fold improvement in the accuracy is consistent with the error terms of the linearization being of order $\varepsilon^{4}$ when considering the square of the $L^2$-norm (because $h_{M+\varepsilon m} = h_M+\varepsilon u_1 + O(\varepsilon^2)$, when we decrease $\varepsilon$ from $1/100$ to $1/1000$, the square of the error term of the linearization is changed by $((1/10)^2)^2=1/1000$). The table also illustrates the change in the norm of the invariant density when perturbed; we see that the norm of the invariant density increases when we perturb $M$ by $\varepsilon m^*$ and decreases when we perturb by $-\varepsilon m^*$, consistent with the choice of sign of $m^*$ noted above.

\begin{table}[h]
\begin{tabular}{|c|c||c|c|c|c|c|}
\hline
$n$ &$ \|u^*_1\|_{L^2}^2$& $\varepsilon$ & $\|h_{M+\varepsilon m^*}-(h_M+\varepsilon u_1^*)\|_{L^2}^2$ & $\|h_{M-\varepsilon m^*}\|^2_{L^2}$ & $\|h_M\|^2_{L^2}$ & $\|h_{M+\varepsilon m^*}\|^2_{L^2}$ \\
\hline
1500 & 0.6180 & 1/100 & 1.3523$\times 10^{-9}$ & 1.007131171 & 1.007824993 & 1.008646526 \\
\hline
 & & 1/1000 & 1.3459$\times 10^{-13}$ & 1.007749900 & 1.007824993 & 1.007901364 \\
\hline
1750 & 0.6165 & 1/100 & 1.3487$\times 10^{-9}$ & 1.007132553 & 1.007825008 & 1.008644876 \\
\hline
 & & 1/1000 & 1.3422$\times 10^{-13}$ & 1.007750064 & 1.007825008 & 1.007901226 \\
\hline
2000 & 0.6154 & 1/100 & 1.3452$\times 10^{-9}$ & 1.007133155 & 1.007825017 & 1.008644069 \\
\hline
 & & 1/1000 & 1.3388$\times 10^{-13}$ & 1.007750143 & 1.007825017 & 1.007901163 \\
\hline
\end{tabular}
\caption{Numerical results for maximising the linear response of the $L^2$-norm of the invariant probability measure of the stochastic Lanford Map.
Column 1: number of Ulam partitions;  Column 2:  optimal objective value;  Column 3: values of $\varepsilon$;  Column 4:  calculation of linearization error;  Columns 5-7:  demonstration that the $L^2$-norm of the invariant density increases and decreases appropriately under the small perturbation $\varepsilon m^*$.}
\label{table-Lanford}
\end{table}

\subsubsection{Maximising the linear response of the expectation of an observable}\label{example-log-fun}
In this section we find the perturbation that generates the greatest linear response of the expectation
\begin{equation*}
\langle c, h\rangle_{L^2} = \int_{[0,1]} c(x)h(x) dx,
\end{equation*}
where $c(x)=2\sin (\pi x)$ and the underlying dynamics are given by the map (\ref{Lanford_Map_Noise}). We consider problem (\ref{obj_lin_fun})-(\ref{zero_lin_fun}) with the vector $\mathbf{c}=(c_1,\dots,c_n)\in\R^n$, where $c_i = \frac{\sqrt{n}}{\|\tilde{\textbf{c}}\|_2}\tilde{c}_i$, $\tilde{c}_i = 2\sin (\pi x_i)$ and $x_i= \frac{i-1}{n}+\frac{1}{2n}$, $i=1,\ldots,n$.
Let $M\in\R^{n\times n}$ be the discretization matrix derived from Ulam's method. We use Algorithm 3 to solve problem (\ref{obj_lin_fun})-(\ref{zero_lin_fun}). Figure \ref{fig-Lanford-noise-fun} presents the optimal perturbation $m^*$ and the associated linear response $u^*_1$ for this problem. Note that the response in Figure \ref{fig-Lanford-noise-fun}(B) has positive values where $c(x)$ is large and negative values where $c(x)$ is small, consistent with our objective to increase the expectation of $c$. 
In this example ($n=2000$), we obtain $\langle c, u^*_1\rangle_{L^2} \approx 0.2514$.
\begin{figure}[h]
\subfloat[Subfigure 1 list of figures text][The optimal perturbation $m^*$. The colourbar indicates the values of the elements of the matrix.]{
\includegraphics[width=0.5\textwidth]{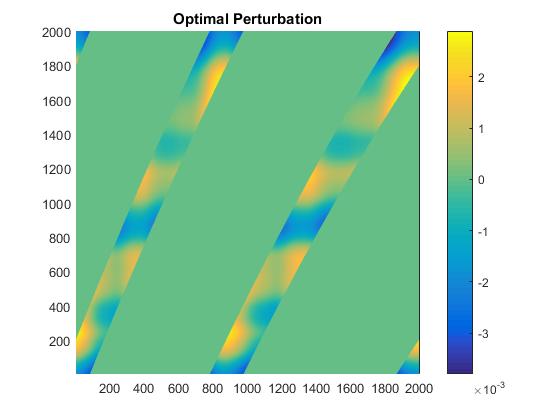}
\label{fig:subfig5}}
\subfloat[Subfigure 2 list of figures text][The optimal linear response $u^*_1$ of the invariant density.]{
\includegraphics[width=0.5\textwidth]{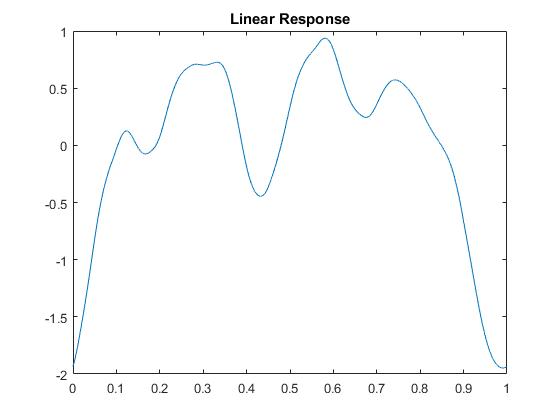}
\label{fig:subfig6}}
\caption{Solution to the problem of maximising the expectation of the response of observable $c(x)$ for the stochastically perturbed Lanford map.}
\label{fig-Lanford-noise-fun}
\end{figure}

Table \ref{table-Lanford-fun} provides numerical results for various partition cardinalities $n$.
We see that (i) the value of $ \langle c, u^*_1\rangle_{L^2}$ appears to converge when we increase $n$, (ii) the 100 fold improvement in accuracy is consistent with the error terms of the linearization being of order $\varepsilon^{2}$ as $h_{M+\varepsilon m} = h_M+\varepsilon u_1 + O(\varepsilon^2)$, and (iii) the expectation increases if we perturb in the direction $\varepsilon m^*$ and decreases if we perturb in the direction $-\varepsilon m^*$.

\begin{table}[h]

\hspace*{-4em}\begin{tabular}{|c|c||c|c|c|c|c|}
\hline
$n$ &$ \langle c, u^*_1\rangle_{L^2}$& $\varepsilon$ & $\langle c, h_{M+\varepsilon m^*}\rangle_{L^2}-\langle c,h_M+\varepsilon u^*_1\rangle_{L^2}$ & $\langle c, h_{M-\varepsilon m^*}\rangle_{L^2} $ & $\langle c, h_M\rangle_{L^2}$ & $\langle c, h_{M+\varepsilon m^*}\rangle_{L^2}$ \\
\hline
1500 & 0.2520 & 1/100 & -9.7038$\times 10^{-6}$ & 0.894337506 & 0.896867102 & 0.899377230 \\
\hline
 & & 1/1000 & -9.7311$\times 10^{-8}$ & 0.896615022 & 0.896867102 & 0.897118988 \\
\hline
1750 & 0.2517 & 1/100 & -9.6835$\times 10^{-6}$ & 0.894340765 & 0.896867054 & 0.899373916 \\
\hline
 & & 1/1000 & -9.7107$\times 10^{-8}$ & 0.896615302 & 0.896867054 & 0.897118612 \\
\hline
2000 & 0.2514 & 1/100 & -9.6677$\times 10^{-6}$ & 0.894343310 & 0.896867024 & 0.899371341 \\
\hline
 & & 1/1000 & -9.6948$\times 10^{-8}$ & 0.896615528 & 0.896867024 & 0.897118325 \\
\hline
5000 & 0.2503 & 1/100 & -9.6076$\times 10^{-6}$ & 0.894354420 & 0.896866939 & 0.899360182 \\
\hline
 & & 1/1000 & -9.6346$\times 10^{-8}$ & 0.896616557 & 0.896866939 & 0.897117127 \\
\hline
7000 & 0.2501 & 1/100 & -9.5961$\times 10^{-6}$ & 0.894356531 & 0.896866931 & 0.899358078 \\
\hline
 & & 1/1000 & -9.6230$\times 10^{-8}$ & 0.896616760 & 0.896866931 & 0.897116909 \\
\hline
\end{tabular}
\caption{Numerical results for maximising the linear response of the expectation of $c(x)=2\sin (\pi x)$ for the stochastic Lanford map. Column 1: number of bins;  Column 2:  optimal objective value;  Column 3: values of $\varepsilon$;  Column 4:  calculation of linearization error;  Columns 5-7:  demonstration that the expected value of the function $c$ increases and decreases appropriately under the small perturbation $\varepsilon m^*$.}
\label{table-Lanford-fun}
\end{table}

\subsection{A stochastically perturbed logistic Map}
In this section, we consider the problems of maximising the $L^2$-norm of the linear response of the invariant measure and maximising the linear response of the expectation of an observable.
The underlying deterministic dynamics is given by the logistic map, and this map is again stochastically perturbed.
\subsubsection{Maximising the linear response of the $L^2$-norm of the invariant measure}
Let $T_\xi:[0,1]\rightarrow [0,1]$ be the logistic map with noise defined by
\begin{equation}\label{Logistic_Map_Noise}
T_{\xi}(x) = 4x(1-x)+\xi_x,
\end{equation}
where
\begin{eqnarray*}
\xi_x\sim
\begin{cases}
\max\{0,\mathcal{U}(x, \frac{1}{10})\}&\text{ if }x<\frac{1}{10}\\
\mathcal{U}(x, \frac{1}{10})&\text{ if }\frac{1}{10}\leq x\leq\frac{9}{10}\\
\min\{\mathcal{U}(x, \frac{1}{10}),1\}&\text{ if }x>\frac{9}{10},
\end{cases}
\end{eqnarray*}
and $\mathcal{U}(x, \frac{1}{10})$ denotes the uniform distribution of radius $1/10$ centred at $x$. Let $M\in\R^{n\times n}$ be Ulam's discretization of the transfer operator of the map $T_\xi$ with $n$ partitions.
     We use Algorithm 2 to solve the optimisation problem (\ref{obj5})-(\ref{zeros}) with the matrix $M$ for $n=2000$ to obtain the optimal perturbation $m^*$. The top two singular values of $\widetilde{U}$, for this example, were computed in MATLAB to be 0.0185 and 0.0147 (each with unit multiplicity); thus, by Proposition \ref{Prop-Unique}, $m^*$ is the unique optimal perturbation. Figure \ref{fig-logistic-noise} shows the results for the stochastically perturbed logistic map; for this example we compute $\|u^*_1\|_{L^2}^2 \approx 0.6815$.
     \begin{figure}[h]
\subfloat[Subfigure 1 list of figures text][Colourmap of the stochastically perturbed logistic map. The colourbar indicates the values of the elements of the matrix.]{
\includegraphics[width=0.5\textwidth]{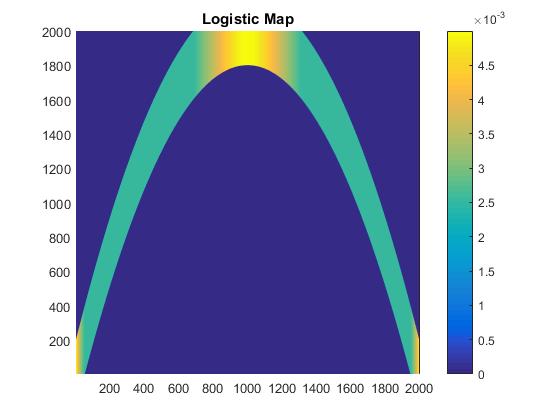}
\label{fig:subfig7}}
\subfloat[Subfigure 2 list of figures text][The invariant density $h$.]{
\includegraphics[width=0.5\textwidth]{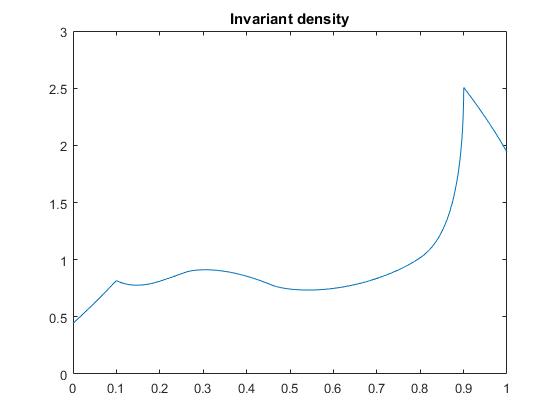}
\label{fig:subfig8}}
\qquad
\subfloat[Subfigure 3 list of figures text][The optimal perturbation $m^*$. The colourbar indicates the values of the elements of the matrix.]{
\includegraphics[width=0.5\textwidth]{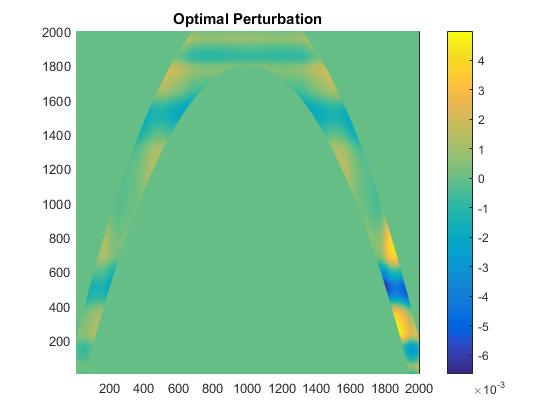}
\label{fig:subfig9}}
\subfloat[Subfigure 4 list of figures text][The optimal linear response $u^*_1$ of the invariant density.]{
\includegraphics[width=0.5\textwidth]{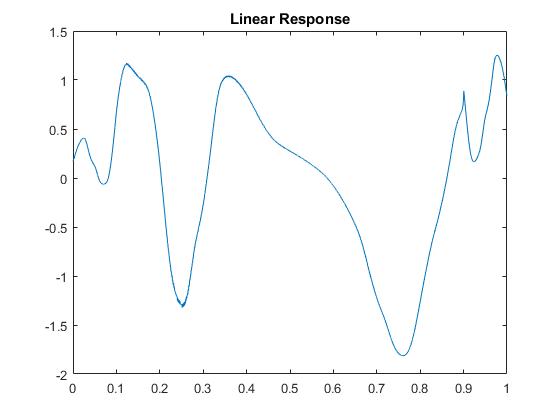}
\label{fig:subfig10}}
\caption{Solution to the problem of maximising the $L^2$-norm of the linear response of the stochastically perturbed logistic map.}
\label{fig-logistic-noise}
\end{figure}
In Figure \ref{fig-logistic-noise}(C), we see sharp increases in mass mapped to neighbourhoods of $x=0.15$ and $x=0.4$, as well as a sharp decrease in mass mapped to a neighbourhood of $x=0.25$;  these observations coincide with the local peaks and troughs of the response vector shown in Figure \ref{fig-logistic-noise}(D).
Table \ref{table-Logistic} displays the corresponding numerical results.

\begin{table}
\begin{tabular}{|c|c||c|c|c|c|c|}
\hline
$n$ &$ \|u^*_1\|_{L^2}^2$& $\varepsilon$ & $\|h_{M+\varepsilon m^*}-(h_M+\varepsilon u_1^*)\|_{L^2}^2$ & $\|h_{M-\varepsilon m^*}\|^2_{L^2}$ & $\|h_M\|^2_{L^2}$ & $\|h_{M+\varepsilon m^*}\|^2_{L^2}$ \\
\hline
1500 & 0.6849 & 1/100 & 7.8493$\times 10^{-10}$ & 1.215630946 & 1.217112326 & 1.218720741 \\
\hline
 & & 1/1000 & 7.8670$\times 10^{-14}$ & 1.216958459 & 1.217112326 & 1.217267464 \\
\hline
1750 & 0.6829 & 1/100 & 7.8297$\times 10^{-10}$ & 1.215635225 & 1.217113142 & 1.218717705 \\
\hline
 & & 1/1000 & 7.8474$\times 10^{-14}$ & 1.216959639 & 1.217113142 & 1.217267913 \\
\hline
2000 & 0.6815 & 1/100 & 7.8099$\times 10^{-10}$ & 1.215637697 & 1.217113684 & 1.218716031 \\
\hline
 & & 1/1000 & 7.8276$\times 10^{-14}$ & 1.216960386 & 1.217113684 & 1.217268245 \\
\hline
\end{tabular}
\caption{Numerical results for maximising the linear response of the $L^2$-norm of the invariant probability measure of the stochastic logistic map.
Column 1: number of Ulam partitions;  Column 2:  optimal objective value;  Column 3: Values of $\varepsilon$;  Column 4:  calculation of linearization error;  Columns 5-7:  demonstration that the $L^2$-norm of the invariant density increases and decreases appropriately under the small perturbation $\varepsilon m^*$.}
\label{table-Logistic}
\end{table}

\subsubsection{Maximising the linear response of the expectation of an observable}
Using (\ref{obj_lin_fun})--(\ref{zero_lin_fun}), we calculate the perturbation achieving a maximal linear response of $\langle c, h\rangle$ for $c(x)=2\sin(\pi x)$ for the stochastic dynamics (\ref{Logistic_Map_Noise}).
We again compute with the vector $\mathbf{c}=(c_1,\dots,c_n)\in\R^n$, where $c_i = \frac{\sqrt{n}}{\|\tilde{\textbf{c}}\|_2}\tilde{c}_i$, $\tilde{c}_i = 2\sin (\pi x_i)$ and $x_i= \frac{i-1}{n}+\frac{1}{2n}$, $i=1,\ldots,n$. We compute the discretization matrix $M\in\R^{n\times n}$ derived from Ulam's method and make use of Algorithm 3.

The $m^*$ provoking the greatest linear response in the expectation $\langle c,h\rangle$ is shown in Figure \ref{fig-logistic-noise-fun} (A).
The linear response corresponding to $m^*$ is shown in Figure \ref{fig-logistic-noise-fun}(B); for this example, $ \langle c, u^*_1\rangle_{L^2}\approx 0.1187$.
The response takes its minimal values at $x=0, x=1$, where the values of the observable $c$ is also least, and the response is broadly positive near the centre of the interval $[0,1]$, where the observable takes on large values;  both of these observations are consistent with maximising the linear response of the observable $c$.


\begin{figure}[h]
\subfloat[Subfigure 1 list of figures text][The optimal perturbation $m^*$. The colourbar indicates the values of the elements of the matrix.]{
\includegraphics[width=0.5\textwidth]{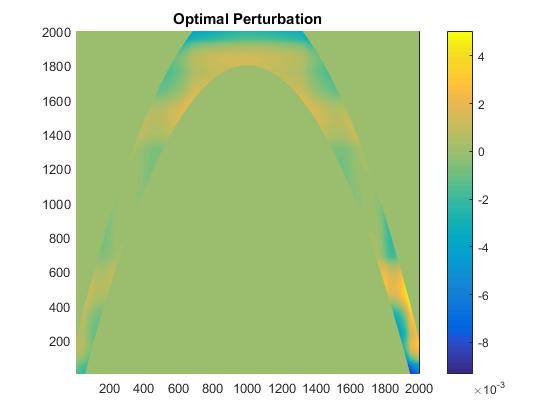}
\label{fig:subfig11}}
\subfloat[Subfigure 2 list of figures text][The optimal linear response $u^*_1$ of the invariant density.]{
\includegraphics[width=0.5\textwidth]{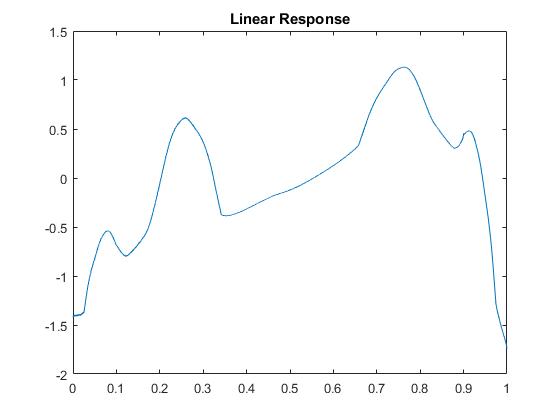}
\label{fig:subfig12}}
\caption{Solution to the problem of maximising the expectation of the response of observable $c(x)$ for the stochastically perturbed logistic map.}
\label{fig-logistic-noise-fun}
\end{figure}

Numerical results for this example are provided in Table \ref{table-Logistic-fun}.

\begin{table}[h]
\hspace*{-4em}\begin{tabular}{|c|c||c|c|c|c|c|}
\hline
$n$ &$ \langle c, u^*_1\rangle_{L^2}$& $\varepsilon$ & $\langle c, h_{M+\varepsilon m^*}\rangle_{L^2}-\langle c,h_M+\varepsilon u^*_1\rangle_{L^2}$ & $\langle c, h_{M-\varepsilon m^*}\rangle_{L^2} $ & $\langle c, h_M\rangle_{L^2}$ & $\langle c, h_{M+\varepsilon m^*}\rangle_{L^2}$ \\
\hline
1500 & 0.1190 & 1/100 & -1.8929$\times 10^{-6}$ & 0.800087366 & 0.801279662 & 0.802468177 \\
\hline
 & & 1/1000 & -1.8903$\times 10^{-8}$ & 0.801160602 & 0.801279662 & 0.801398684 \\
\hline
1750 & 0.1189 & 1/100 & -1.8871$\times 10^{-6}$ & 0.800089179 & 0.801279736 & 0.802466524 \\
\hline
 & & 1/1000 & -1.8845$\times 10^{-8}$ & 0.801160850 & 0.801279736 & 0.801398585 \\
\hline
2000 & 0.1187 & 1/100 & -1.8845$\times 10^{-6}$ & 0.800090673 & 0.801279783 & 0.802465129 \\
\hline
 & & 1/1000 & -1.8819$\times 10^{-8}$ & 0.801161041 & 0.801279783 & 0.801398487 \\
\hline
5000 & 0.1182 & 1/100 & -1.8705$\times 10^{-6}$ & 0.800096427 & 0.801279916 & 0.802459669 \\
\hline
 & & 1/1000 & -1.8679$\times 10^{-8}$ & 0.801161734 & 0.801279916 &  0.801398059 \\
\hline
7000 & 0.1181 & 1/100 & -1.8678$\times 10^{-6}$ & 0.800097497 & 0.801279928 & 0.802458629 \\
\hline
 & & 1/1000 & -1.8652$\times 10^{-8}$ & 0.801161852 & 0.801279928 & 0.801397966 \\
\hline
\end{tabular}
\caption{Numerical results for maximising the linear response of the expectation of $c(x)=2\sin(\pi x)$ for the stochastic logistic map.
Column 1: number of bins;  Column 2:  optimal objective value;  Column 3: values of $\varepsilon$;    Column 4:  calculation of linearization error;  Columns 5-7:  demonstration that the expected value of the function $c$ increases and decreases appropriately under the small perturbation $\varepsilon m^*$.}
\label{table-Logistic-fun}
\end{table}

\subsection{Double Lanford Map}\label{sect:doublelanford}
In this last section, we consider the problem of maximising the linear response of the rate of convergence to the equilibrium. The underlying deterministic dynamics is given by a stochastically perturbed double Lanford map. More explicitly, we consider the map $T:S^1\rightarrow S^1$ defined by
\begin{equation}
T(x) =
\begin{cases}
\left(T_{Lan}(2x)\mod \frac{1}{2}\right)+\xi\mod 1 & \text{ if } 0\leq x\leq \frac{1}{2}\\
\left(T_{Lan}\left(2\left(x-\frac{1}{2}\right)\right)\mod \frac{1}{2}\right)+\frac{1}{2}+\xi\mod 1 & \text{ if } \frac{1}{2}<x\leq 1,
\end{cases}
\end{equation}
where $T_{Lan}(x) = 2x+\frac{1}{2}x(1-x)$ and $\xi\sim\mathcal{U}(0,\frac{1}{10})$.
We have chosen this doubled version of the Lanford map in order to study a relatively slowly (but still exponentially) mixing system.
The subintervals $[0,1/2]$ and $[1/2,1]$ are ``almost-invariant'' because there is only a relatively small probability that points in each of these subintervals are mapped into the complementary subinterval;  see Figure \ref{fig-lanford-double-noise}(A).

\begin{figure}[h]
\subfloat[Subfigure 1 list of figures text][Colourmap of the stochastically perturbed double Lanford map. The colourbar indicates the values of the elements of the matrix.]{
\includegraphics[width=0.5\textwidth]{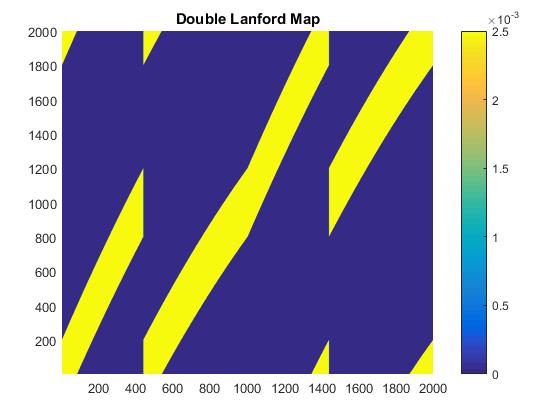}
\label{fig:subfig13}}
\subfloat[Subfigure 2 list of figures text][The invariant density $h$.]{
\includegraphics[width=0.5\textwidth]{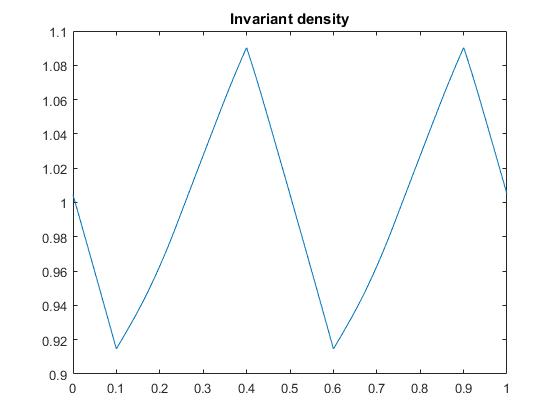}
\label{fig:subfig14}}
\qquad
\subfloat[Subfigure 3 list of figures text][The optimal perturbation $m^*$. The colourbar indicates the values of the elements of the matrix.]{
\includegraphics[width=0.5\textwidth]{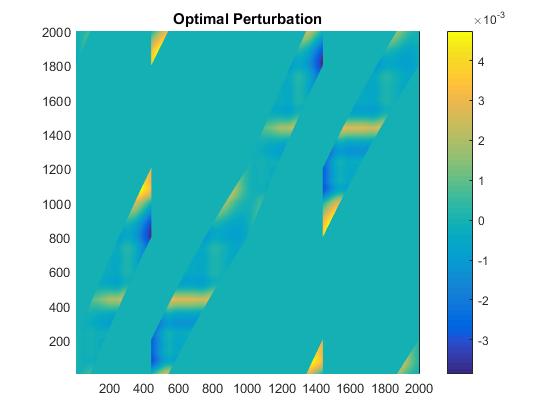}
\label{fig:subfig15}}
\subfloat[Subfigure 4 list of figures text][The optimal linear response $u^*_1$ of the invariant density.]{
\includegraphics[width=0.5\textwidth]{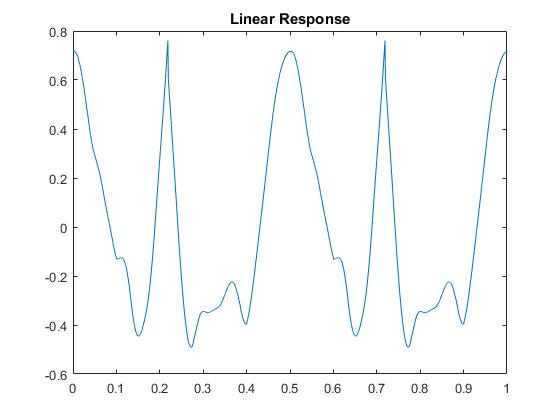}
\label{fig:subfig16}}
\caption{Solution to the problem of maximising the linear response of the rate of convergence to the equilibrium of the stochastically perturbed double Lanford map.}
\label{fig-lanford-double-noise}
\end{figure}

Let $M\in\R^{n\times n}$ be Ulam's discretization of the transfer operator of the map $T$ with $n$ partitions. Using Algorithm 4, we solve problem (\ref{obj_eval2})-(\ref{zero_lin_fun_eval2}) for the matrix $M$ for $n=2000$. Figure \ref{fig-lanford-double-noise} shows the double Lanford map and the approximation of the invariant density $h$ of this map. Figure \ref{fig-lanford-double-noise}(C) shows the optimal perturbation matrix $m^*$ that maximises the linear response of the rate of convergence to the equilibrium and Figure \ref{fig-lanford-double-noise}(D) shows the corresponding linear response $u_1^*$ of the invariant density $h$. We note that the sign of the matrix $m^*$ is chosen such that the $\nu$ in ~\eqref{70b} is negative.
The optimal objective is given by $\rho = \frac{d(\Re(\log\lambda_2(\varepsilon)))}{d\varepsilon}|_{\varepsilon = 0} \approx -0.2843$.
Figure \ref{fig-lanford-double-noise}(C) shows that most of the large positive values in the perturbation occur in the upper left and lower right blocks of the graph of the double Lanford map, precisely to overcome the almost-invariance of the subintervals $[0,1/2]$ and $[1/2,1]$.
In order to compensate for these increases, there are commensurate negative values in the lower left and upper right.
The net effect is that more mass leaves each of the almost-invariant sets at each iteration of the stochastic dynamics, leading to an increase in mixing rate.

Table \ref{table-Lanford-double} illustrates the numerical results. The value of $\rho$, namely the estimated derivative of the real part of $\log(\lambda_2)$, minimised over all valid perturbations, is shown in the second column. As $n$ increases, $\rho$ appears to converge to a fixed value. Let $r$ and $l$ denote the left and right eigenfunctions of $\mathcal{P}$ corresponding to the second largest eigenvalue, $\delta\mathcal{P}$ denote the perturbation operator approximated by $m^*$, and $\eta_2= \langle l, \delta\mathcal{P} (r)\rangle_{L^2}$, the analogue of ~\eqref{dlambda} in the continuous setting. In the fourth column, we see that the absolute value of the linearization of the perturbed eigenvalue, $|\lambda_2+\varepsilon \eta_2|$, is close to the absolute value of the optimally perturbed eigenvalue, $|\lambda_2(\varepsilon)^*|$. Finally, to verify the parity of $m^*$ is correct, in Table 5 we observe that the absolute value of the second eigenvalue increases when we perturb in the direction $-\varepsilon m^*$ and decreases as we perturb in the direction $\varepsilon m^*$, as required for the perturbation to increase the mixing rate.

\begin{table}[h]
\hspace*{-4em}\begin{tabular}{|c|c||c|c|c|c|c|}
\hline
$n$ &$\rho$ & $\varepsilon$ & $|\lambda_2(\varepsilon)^*|-|\lambda_2+\varepsilon \eta_2|$ & $|\lambda_2(-\varepsilon)^*|$ & $|\lambda_2|$ & $|\lambda_2(\varepsilon)^*|$ \\
\hline
1500 & -0.2852 & 1/100 & -4.2129$\times 10^{-5}$ & 0.849558095 & 0.847154908 &0.844725328  \\
\hline
 & & 1/1000 & -4.4851$\times 10^{-7}$ & 0.847396407 & 0.847154908 & 0.846913145 \\
\hline
1750 & -0.2846 & 1/100 & -4.1719$\times 10^{-5}$ & 0.849553120 & 0.847155348&0.844731281  \\
\hline
 & & 1/1000 & -4.6674$\times 10^{-7}$ & 0.847396301 & 0.847155348 &0.846914132  \\
\hline
2000 & -0.2843 & 1/100 & -4.2606$\times 10^{-5}$ & 0.849550779 & 0.847155633 &0.844734275  \\
\hline
 & & 1/1000 & -5.5723$\times 10^{-7}$ & 0.847396320 & 0.847155633 & 0.846914684 \\
\hline
5000 & -0.2823 & 1/100 & -3.9567$\times 10^{-5}$ & 0.849535385 & 0.847156392 &0.844751481  \\
\hline
 & & 1/1000 & -4.1491$\times 10^{-7}$ & 0.847395450 & 0.847156392 & 0.846917075 \\
\hline
7000 & -0.2820 & 1/100 & -3.9229$\times 10^{-5}$ & 0.849532569 & 0.847156528 &0.844754619  \\
\hline
 & & 1/1000 & -4.0689$\times 10^{-7}$ & 0.847395289 & 0.847156528  & 0.846917509 \\
\hline
\end{tabular}
\caption{Numerical results for the double Lanford Map. Column 1: number of bins;  Column 2:  optimal objective value;  Column 3: values of $\varepsilon$;  Column 4:  calculation of linearization error;  Columns 5-7:  demonstration that the absolute value of the second eigenvalue increases and decreases appropriately under the small perturbation $\varepsilon m^*$.}
\label{table-Lanford-double}
\end{table}

\section*{Acknowledgements}
The authors thank Guoyin Li for helpful remarks concerning Section 3 and Jeroen Wouters for some literature suggestions.
FA is supported by the Australian Government's Research Training Program. DD is supported by an ARC Discovery Project DP150100017 and received partial support from the Croatian Science Foundation under the grant IP-2014-09-2285. GF is partially supported by DP150100017.

\clearpage
\bibliographystyle{plain}

\end{document}